\documentclass[a4paper,11pt]{article}
\usepackage{amsmath,amsfonts,amsthm,amssymb,thm-restate}
\usepackage{pstricks}
\usepackage[labelsep=none]{caption}
\newtheorem{theorem}{Theorem}[section]
\newtheorem{corollary}[theorem]{Corollary}
\newtheorem{lemma}[theorem]{Lemma}
\newtheorem{proposition}[theorem]{Proposition}

\theoremstyle{definition}
\newtheorem{definition}[theorem]{Definition}
\newtheorem{remark}[theorem]{Remark}

\newcommand{\intvl}{\mathrm{I}}              	
\newcommand{\sphere}{\mathbb{S}^2}  					
\newcommand{\Sphere}{\mathbb{S}^3}  					
\DeclareMathOperator*{\nhd}{\mathcal{N}} 			
\DeclareMathOperator*{\Int}{int}  						
\DeclareMathOperator{\lk}{lk}									
\DeclareMathOperator{\ms}{MS}    							
\DeclareMathOperator{\is}{IS}    							
\DeclareMathOperator*{\dist}{d}   						
\DeclareMathOperator{\V}{V}   								

\begin{document}
\title{The Kakimizu complex of a connected sum of links}
\author{Jessica E. Banks}
\date{}
\maketitle

\begin{abstract}
We show that $|\ms(L_1\# L_2)|=|\ms(L_1)|\times|\ms(L_2)|\times\mathbb{R}$ when $L_1$ and $L_2$ are any non-split and non-fibred links. Here $\ms(L)$ denotes the Kakimizu complex of a link $L$, which records the taut Seifert surfaces for $L$. We also show that the analogous result holds if we study incompressible Seifert surfaces instead of taut ones.
\end{abstract}

\section{Introduction}

The Kakimizu complex of a link $L$, denoted $\ms(L)$, is a simplicial complex that records the taut Seifert surfaces for $L$. It is analogous to the curve complex of a compact, orientable surface. We will give the definition in Section \ref{msdefnsection}.

Kakimizu proved the following, building on work of Eisner in \cite{MR0440528}.

\begin{theorem}[\cite{MR1177053} Theorem B]
Let $L_1,L_2$ be knots, each not fibred but with a unique incompressible Seifert surface. Then $|\ms(L_1\#L_2)|$ is homeomorphic to $\mathbb{R}$.
\end{theorem}

In this paper we prove the following more general result, as well as the analogous result for incompressible Seifert surfaces.

\begin{restatable}{theorem}{connectsumthm}\label{connectsumthm}
Let $L_1,L_2$ be non-split, non-fibred, links in $\Sphere$, 
and let $L=L_1\#L_2$. Then $|\ms(L)|$ is homeomorphic to $|\ms(L_1)|\times|\ms(L_2)|\times\mathbb{R}$.
\end{restatable}

For a distant union of two links $L_1$ and $L_2$, the taut Seifert surfaces for the two links do not interact. We may therefore consider the two links separately. For this reason, the `non-split' hypothesis in Theorem \ref{connectsumthm} is not a significant restriction. For the remainder of this paper we will only consider non-split links.

The case that one of the $L_i$ is fibred was dealt with by Kakimizu.

\begin{proposition}[\cite{MR2131376} Proposition 2.4] 
If $L_1$ is fibred then $\ms(L)\cong\ms(L_2)$.
\end{proposition}

The idea of the proof of Theorem \ref{connectsumthm} is to first define a triangulation of $|\ms(L_1)|\times|\ms(L_2)|\times\mathbb{R}$, then choose a representative of each element of $\V(\ms(L_i))$, which we use to define a map between the vertices of the two simplicial complexes, and finally show that this map is an isomorphism.

Sections \ref{msdefnsection}--\ref{orderingsection} cover definitions and results we will need. Sections \ref{repssection}--\ref{edgessection} each constitute a step in the proof of Theorem \ref{connectsumthm}. In Section \ref{proofsection} these ideas are drawn together to complete the proof.
Finally, in Section \ref{incompressiblesection} we discuss incompressible Seifert surfaces.

I am grateful to Marc Lackenby for suggesting that this proof might be possible. This result has been proved independently by Bassem Saad.

\section{The definition of the Kakimizu complex}\label{msdefnsection}

\begin{definition}
A \textit{Seifert surface} for a link $L$ is a compact, orientable surface $R$ with no closed components such that $\partial R=L$ as an oriented link. The surface $R$ can also be seen as properly embedded in the link complement $\Sphere\setminus\nhd(L)$.
Say $R$ is \textit{taut} if it has maximal Euler characteristic among all Seifert surfaces for $L$.
\end{definition}

The Kakimizu complex was first defined as follows.

\begin{definition}[\cite{MR1177053} p225]\label{firstmsdefn}
For a link $L$, the \textit{Kakimizu complex} $\ms(L)$ of $L$ is a simplicial complex, the vertices of which are the ambient isotopy classes of taut Seifert surfaces for $L$. Vertices $R_0,\ldots,R_n$ span an $n$--simplex exactly when they can be realised disjointly.
\end{definition}

\begin{definition}
A metric is defined on the vertices of $\ms(L)$. The distance between two vertices is the distance in the 1--skeleton of $\ms(L)$ when every edge has length 1.
\end{definition}

In \cite{MR1177053}, Kakimizu showed that the distance between two vertices of $\ms(L)$ can be calculated by considering lifts of the two Seifert surfaces to the infinite cyclic cover of the link complement.

\begin{proposition}[\cite{MR1177053} Proposition 3.1]\label{msdistanceprop}
Let $L$ be a link, and let $M=\Sphere\setminus\nhd(L)$. Consider the infinite cyclic cover $\tilde{M}$ of $M$ (that is, the cover corresponding to the linking number $\lk\colon\pi_1(M)\to\mathbb{Z}$), and let $\tau$ be a generator for the group of covering transformations.

Let $R,R'$ be taut Seifert surfaces for $L$ that represent distinct vertices of $\ms(L)$. Choose a lift $V_0$ of $M\setminus R$ to $\tilde{M}$. For $n\in\mathbb{Z}$ let $V_n=\tau^n(V_0)$.

Take a lift $V_{R'}$ of $M\setminus R'$. Isotope the Seifert surface $R'$ in $M$ so as to minimise the value of $\max\{n:V_{R'}\cap V_n\neq\emptyset\}-\min\{n:V_{R'}\cap V_n\neq\emptyset\}$. Then $\dist_{\ms(L)}(R,R')=\max\{n:V_{R'}\cap V_n\neq\emptyset\}-\min\{n:V_{R'}\cap V_n\neq\emptyset\}$.
\end{proposition}

Przytycki and Schultens pointed out in \cite{przytycki-2010} that this result is not in fact true. It may fail in the case of a link that bounds a disconnected taut Seifert surface. Such links are called \textit{boundary links}. Note that this is a fairly unusual property for a link to have. In particular, all boundary links have Alexander polynomial 0. However, we will need to allow for this case in the proof of Theorem \ref{connectsumthm}, which makes the proof more complicated than it would be otherwise. In particular, it is difficult to control the situation where pairs of Seifert surfaces have components that can be made to coincide.

A reader who wishes to avoid this difficulty should take Definition \ref{firstmsdefn} as their working definition of $\ms(L)$, and may ignore the remainder of this section, Section \ref{mfldssection} from Lemma \ref{arcintersectionlemma} onwards and Lemma \ref{threearcslemma}, as well as some of the detail in the proofs of Corollary \ref{edgescor} and Proposition \ref{mainedgesprop}.

Figure \ref{msdefnpic1} gives an instructive example of how Proposition \ref{msdistanceprop} can fail. The link $L_{\alpha}$ shown consists of two linked copies of the knot $7_4$. This knot, which is the $(1,3,3)$ pretzel knot, has two taut Seifert surfaces. One is given by applying Seifert's algorithm to an alternating diagram. The other is given by performing a flype on the diagram that moves the single crossing across one of the lines of three crossings, and then applying Seifert's algorithm. Combining two copies of each of these surfaces gives the two disjoint taut Seifert surfaces $R_a,R_b$ for $L_{\alpha}$ shown in Figure \ref{msdefnpic1}. The arc $\rho$ shown in Figure \ref{msdefnpic2} runs from the top side of $R_a$ to the bottom side, and in doing so passes through $R_b$ in the positive direction twice. This means that if we first take lifts $V_n$ of $M\setminus R_b$ and then a lift $V$ of $M\setminus R_a$ as in Proposition \ref{msdistanceprop} we find that, for example, $V$ intersects $V_0,V_1,V_2$. The surfaces $R_a,R_b$ should therefore be distance 2 apart, instead of adjacent.
\begin{figure}[htbp]
\centering
\input{pictexfiles/msdefnpic1}
\caption{\label{msdefnpic1}}
\end{figure}
\begin{figure}[htbp]
\centering
\input{pictexfiles/msdefnpic2}
\caption{\label{msdefnpic2}}
\end{figure}

With Proposition \ref{msdistanceprop} and their own work in mind, Przytycki and Schultens redefine the Kakimizu complex as follows.

\begin{definition}[see \cite{przytycki-2010}]\label{fullmsdefn}
Let $L$ be a link, and let $M=\Sphere\setminus\nhd(L)$. Define the \textit{Kakimizu complex} $\ms(L)$ of $L$ to be the following flag simplicial complex. Its vertices are ambient isotopy classes of taut Seifert surfaces for $L$. Two vertices span an edge if they have representatives $R,R'$ such that a lift of $M\setminus R'$ to the infinite cyclic cover of $M$ intersects exactly two lifts of $M\setminus R$.
\end{definition}

\section{Products of simplicial complexes}\label{complexessection}

\begin{definition}[\cite{MR0050886} Chapter II Definition 8.7]
A simplicial complex $\mathcal{X}$ is \textit{ordered} if there is a binary relation $\leq$ on the vertices of $\mathcal{X}$ with the following properties.
\begin{itemize}
 \item[(P1)] $(u \leq v \textrm{ and } v\leq u)\Rightarrow u=v$.
 \item[(P2)] If $u,v$ are distinct, $(u \leq v \textrm{ or } v\leq u) \Leftrightarrow \textrm{ $u$ and $v$ are adjacent}$.
 \item[(P3)] If $u,v,w$ are vertices of a 2--simplex then $(u \leq v \textrm{ and } v\leq w)\Rightarrow u\leq w$.
\end{itemize}
\end{definition}

\begin{remark}
It is clear that in searching for such a relation we may use the following weaker version of (P2).
\begin{itemize}
 \item[(P2)$'$] If $u,v$ are adjacent then $(u\leq v \textrm{ or } v\leq u)$. 
\end{itemize}
To see this, note that we can remove all relationships between non-adjacent vertices.
\end{remark}

\begin{definition}[\cite{MR0050886} Chapter II Definition 8.8]
Let $\mathcal{X}_1,\mathcal{X}_2$ be ordered simplicial complexes. We define the simplicial complex $\mathcal{X}_1\times\mathcal{X}_2$. Its vertices are given by the set $\V(\mathcal{X}_1)\times\V(\mathcal{X}_2)$. Vertices $(u_0,v_0),\ldots,(u_n,v_n)$ span an $n$--simplex if the following hold.
\begin{itemize}
 \item $\{u_0,\cdots,u_n\}$ is an $m$--simplex of $\mathcal{X}_1$ for some $m\leq n$.
 \item $\{v_0,\cdots,v_n\}$ is an $m$--simplex of $\mathcal{X}_2$ for some $m\leq n$.
 \item The relation defined by $(u,v)\leq (u',v')\Leftrightarrow(u\leq u' \textrm{ and } v\leq v')$ gives a total linear order on $(u_0,v_0),\ldots,(u_n,v_n)$.
\end{itemize}
\end{definition}

\begin{remark}
The projection maps on the vertices extend to simplicial maps of the complexes.
\end{remark}

\begin{theorem}[\cite{MR0050886} Chapter II Lemma 8.9]\label{productcomplextheorem}
The map $|\mathcal{X}_1\times\mathcal{X}_2|\to|\mathcal{X}_1|\times|\mathcal{X}_2|$ induced by projection is a homeomorphism.
\end{theorem}

\begin{definition}
Denote by $\mathcal{Z}$ the simplicial complex with a vertex at each integer and an edge joining $n-1$ to $n$ for each $n\in\mathbb{Z}$, so that $|\mathcal{Z}|\cong\mathbb{R}$. Order the vertices of $\mathcal{Z}$ using the usual order $\leq$ on $\mathbb{Z}$.
\end{definition}

\section{Product regions and detecting adjacency}\label{mfldssection}

In this section we recall some results from 3--manifold theory that we will need. These include a number of related results; we include full proofs as later proofs are sensitive to the details. The concepts covered in this section are generally well-known, although some of the details are specific to the case at hand.

\begin{definition}[\cite{przytycki-2010} Section 3]
Let $M$ be a connected 3--manifold, and let $S,S'$ be (possibly disconnected) surfaces properly embedded in $M$. 

Call $S$ and $S'$ \textit{almost transverse} if, given a component $S_0$ of $S$ and a component $S'_0$ of $S'$, they either coincide or intersect transversely.
Call the surfaces \textit{almost disjoint} if, given a component $S_0$ of $S$ and a component $S'_0$ of $S'$, they either coincide or are disjoint.
Say they are \textit{$\partial$--almost disjoint} if $\partial S=\partial S'$ and, given a component $S_0$ of $S$ and a component $S'_0$ of $S'$, they either coincide or have disjoint interiors.

Say $S$ and $S'$ \textit{bound a product region} if the following holds. There is a compact surface $T$, a finite collection $\rho_T\subseteq \partial T$ of arcs and simple closed curves and a map of $N=(T\times\intvl)/\!\sim$ into $M$ that is an embedding on the interior of $N$ and 
has the following properties. 
\begin{itemize}
 \item $T\times\{0\}=S\cap N$ and $T\times\{1\}=S'\cap N$.
 \item $\partial N\setminus (T\times\partial\intvl)\subseteq\partial M$.
\end{itemize}
Here $\sim$ collapses $\{x\}\times\intvl$ to a point for each $x\in\rho_T$.
The \textit{horizontal boundary} of $N$ is $(T\times\partial\intvl)/\!\sim$.
Say $S$ and $S'$ have \textit{simplified intersection} if they do not bound a product region.
\end{definition}

\begin{proposition}[\cite{MR1315011} Proposition 4.8; see also \cite{MR0224099} Proposition 5.4 and Corollary 3.2]\label{productregionsprop}
Let $M$ be a $\partial$--irreducible Haken manifold.
Let $S,S'$ be incompressible, $\partial$--incompressible surfaces properly embedded in $M$ in general position. 
\begin{enumerate}
 \item If $S$ and $S'$ are isotopic then there is a product region between them.
 \item Suppose $S\cap S'\neq\emptyset$, but $S$ can be isotoped to be disjoint from $S'$. Then there is a product region between $S$ and $S'$.
\end{enumerate}
\end{proposition}

\begin{remark}
We will usually apply this proposition with $M=\Sphere\setminus\nhd(L)$ for a link $L$. If $L$ is neither a split link nor the unknot then $M$ is Haken and $\partial$--irreducible. Furthermore, if $S,S'$ are taut Seifert surfaces for $L$ then they are properly embedded, incompressible and $\partial$--incompressible. This remains true if we only consider some components of such surfaces. 
\end{remark}

\begin{remark}
If $\partial S$ and $\partial S'$ are transverse, the product region $N$ given by Proposition \ref{productregionsprop} is always embedded in $M$.
However, we will want to apply the proposition when Seifert surfaces $S$ and $S'$ may have components that either coincide or have boundaries that coincide. It continues to hold in this situation, but may result in a product region that is not embedded. 
There are two ways that this can occur. 

One option is that the components $S_0$ of $S$ and $S'_0$ of $S'$ that bound $N$ coincide. Then $N$ is the whole of $M$, the surfaces $S$ and $S'$ are connected, and the link $L$ is fibred with fibre $S$.
As we are only interested in non-fibred links, this case will not arise.

The second possibility is that $S_0$ and $S'_0$ do not coincide but their boundaries do.
Then $\partial N$ covers an entire component of $\partial\!\nhd(L)$, and $T\times\{0\}$ meets $T\times\{1\}$ along at least one component of $\partial T$.
See Figure \ref{edgespic1}.
\end{remark}

\begin{corollary}\label{isotopyrelbdycor}
Suppose $L$ is not fibred. If $S,S'$ are isotopic by an isotopy fixing their boundaries then there is a product region $N=(T\times\intvl)/\!\sim$ between them with $\rho_T=\partial T$. 
\end{corollary}
\begin{proof}
Suppose no such product region exists. By Proposition \ref{productregionsprop} there is a product region $N$ between $S$ and $S'$. This product region $N$ meets one component of each of $S$ and $S'$, and the boundaries of these components coincide. By deleting other components if necessary, we may assume $S,S'$ are connected. 

Let $\tilde{S}$ be a lift of $S$ to the infinite cyclic cover. The isotopy from $S$ to $S'$ lifts to an isotopy from $\tilde{S}$ to a lift $\tilde{S}'$ of $S'$. Note that $\partial\tilde{S}=\partial\tilde{S}'$. By hypothesis, the isotopy from $S'$ to $S$ defined by the product region $N$ moves the boundary of $S'$. As the boundaries of $S$ and $S'$ coincide, this isotopy therefore takes each component of $\partial S'$ once around the torus component of $\partial M$ on which it lies. Hence the isotopy defined by $N$ lifts to an isotopy from $\tilde{S}'$ to either $\tau(\tilde{S})$ or $\tau^{-1}(\tilde{S})$. Thus $\tilde{S}$ is isotopic to $\tau(\tilde{S})$. Again by Proposition \ref{productregionsprop} there is a product region between $\tilde{S}$ and $\tau(\tilde{S})$. This contradicts that $L$ is not fibred.
\end{proof}

\begin{remark}
The condition that $L$ is not fibred is not actually necessary for this result, only for our proof.
\end{remark}

\begin{proposition}[\cite{MR0224099} Corollary 3.2]\label{surfaceinproductprop}
Suppose surfaces $S_0,S_1$ bound a product region $N$. Let $S'$ be an incompressible surface that is transverse to $S_0,S_1$. Suppose $S'\cap \Int(N)\neq\emptyset$ but $S'\cap (S_1\cap N)=\emptyset$. Then each component of $S'\cap\Int(N)$ bounds a product region with a subsurface of $S_0$. In particular, if additionally $S'\cap (S_0\cap N)=\emptyset$ then this component of $S'$ is parallel to those of $S_0,S_1$ that bound $N$.
\end{proposition}

\begin{proposition}[\cite{przytycki-2010} Corollary 3.4]\label{disjointinteriorsprop}
Let $M_a,M_b$ be proper 3--submanifolds of $\tilde{M}$ such that $\partial M_a$ and $\partial M_b$ are unions of lifts of minimal genus Seifert surfaces that are almost transverse with simplified intersection. If $M_a$ can be isotoped to have interior disjoint from $M_b$ then the interior of $M_a$ is disjoint from $M_b$.
\end{proposition}

\begin{lemma}[\cite{przytycki-2010} Lemma 3.5]\label{maketransverselemma}
Let $R_1,R_2,R_3$ be minimal genus Seifert surfaces. Then they can be isotoped to be pairwise almost transverse and have pairwise simplified intersection.
\end{lemma}

\begin{proposition}[\cite{przytycki-2010} Proposition 3.2]\label{realisedistanceprop}
In the notation of Proposition \ref{msdistanceprop}, if $R'$ is almost transverse to and has simplified intersection with $R$ then it minimises $\max\{n:V_{R'}\cap V_n\neq\emptyset\}-\min\{n:V_{R'}\cap V_n\neq\emptyset\}$.
\end{proposition}

The following criterion allows us to test for adjacency under Definition \ref{fullmsdefn}.

\begin{lemma}\label{arcintersectionlemma}
Let $R_1, R_2$ be fixed, almost disjoint, taut Seifert surfaces for a link $L$. Then $R_1,R_2$ demonstrate that their isotopy classes are at most distance 1 apart in $\ms(L)$ if and only if the following holds for all pairs $(x,y)$ of points of $R_1\setminus R_2$. 

Choose a product neighbourhood $\nhd(R_1)=R_1\times[-1,1]$ in $M$ for $R_1$ with $R_1=R_1\times\{0\}$, such that $R_1\times\{1\}$ lies on the positive side of $R_1$. Let $\rho\colon \intvl\to M$ be any path with $\rho(0)=(x,1)$ and $\rho(1)=(y,-1)$ that is disjoint from $R_1$ and transverse to $R_2$. 
Then the algebraic intersection number of $\rho$ and $R_2$ is 1.
\end{lemma}
\begin{proof}
Suppose the condition holds for all pairs $(x,y)$. If $R_1,R_2$ coincide everywhere then they are distance 0 apart. Assume otherwise, and let $x_0$ be a point of $R_1\setminus R_2$. Choose a lift $V_{R_2}$ of $M\setminus R_2$, and let $\tilde{x}_0$ be the lift of $x_0$ that lies in $V_{R_2}$. Let $V_{R_1}$ be the lift of $M\setminus R_1$ that lies above $\tilde{x}_0$, and let $\tilde{R}_1$ be the lift of $R_1$ that lies between $V_{R_1}$ and $\tau^{-1}(V_{R_1})$. Then $\tilde{x}_0$ lies on $\tilde{R}_1$, and $V_{R_1}$ meets $V_{R_2}$ and $\tau(V_{R_2})$. We wish to show that these are the only lifts of $M\setminus R_2$ that $V_{R_1}$ meets. Suppose otherwise. 

First suppose that $V_{R_1}$ meets $\tau^{-1}(V_{R_2})$. Then there is a path $\rho$ in $M$ that runs from just above $x_0$ to the projection of a point in $V_{R_1}\cap\tau^{-1}(V_{R_2})$, that is disjoint from $R_1$ and that has algebraic intersection $-1$ with $R_2$. There is also a path $\rho'$ from this point back to above $x_0$ that is disjoint from $R_2$. We may assume both paths are transverse to $R_2$. This forms a closed curve that has algebraic intersection $-1$ with $R_2$. It therefore has algebraic intersection $-1$ with $R_1$. 

Consider the first point $x_1$ at which $\rho'$ crosses $R_1$. Then $x_1\in R_1\setminus R_2$. If $\rho'$ passes through $R_1$ in the positive direction at $x_1$ then the section of the path $\rho\cup\rho'$ from $(x_0,1)$ to $(x_1,-1)$ contradicts the hypothesis as it has algebraic intersection $-1$ with $R_2$ instead of $1$.
Thus $\rho'$ passes through $R_1$ in the negative direction at $x_1$. Stop the path just above $x_1$, at $(x_1,1)$, and add in a path that runs to $(x_1,-1)$ in $M\setminus R_1$. This final section of path has algebraic intersection $1$ with $R_2$ by the hypothesis. Then the complete path gives a contradiction with the hypothesis, as it has algebraic intersection 0 with $R_2$.
Hence $V_{R_1}$ lies entirely above $\tau^{-1}(V_{R_2})$.

Now suppose $V_{R_1}$ meets $\tau^2(V_{R_2})$. Then there is a path $\rho$, disjoint from $R_1$, from $(x_0,1)$ to the projection of a point in $V_{R_1}\cap\tau^2(V_{R_2})$ that has algebraic intersection $2$ with $R_2$. Again, there is a second path $\rho'$ from there to $(x_0,1)$ that is disjoint from $R_2$. The closed curve has algebraic intersection 2 with $R_2$ and $R_1$. Consider the first point $x_1$ at which $\rho'$ crosses $R_1$. Then $x_1\in R_1\setminus R_2$. If $\rho'$ passes through $R_1$ in the positive direction at $x_1$ then the path up to this point contradicts the hypothesis, as it has intersection 2 with $R_2$ instead of 1. Thus it passes through $R_1$ in the negative direction. Stop the path just above $x_1$, at $(x_1,1)$, and add in a path that runs to $(x_1,-1)$ in $M\setminus R_1$. This final section of path has algebraic intersection $1$ with $R_2$ by the hypothesis. Thus the complete path gives a contradiction with the hypothesis, as it has intersection 3 with $R_2$.
Hence we have the required result.

Conversely, suppose any lift of $M\setminus R_1$ meets at most two lifts of $M\setminus R_2$. Choose $x,y,\rho$ as in the condition to be checked. Take the lift $\tilde{\rho}$ of $\rho$ with $\tilde{\rho}(0)$ in a fixed lift $V_{R_2}$ of $M\setminus R_2$. We may use $\tilde{\rho}$ in defining the lift $V_{R_1}$ of $M\setminus R_1$. Therefore $\tilde{\rho}$ is contained in two lifts of $M\setminus R_2$ and the lift of $R_2$ between them. One of these two lifts is $V_{R_2}$, since this contains the lift $\tilde{\rho}(0)$ of $(x,1)$. In addition, the lift of $(x,-1)$ lies in $\tau(V_{R_2})$. Thus the lift  $\tilde{\rho}(1)$ of $(y,-1)$ is in $V_{R_2}$ or in $\tau(V_{R_2})$. We must show that it is in $\tau(V_{R_2})$. If not then it lies in $V_{R_2}$. Then the lift of $(y,1)$ lies in $\tau^{-1}(V_{R_2})$, which is a contradiction.
\end{proof}

\begin{lemma}\label{almostdisjointlemma}
Let $R_a$ be a taut Seifert surface for $L$.
Suppose $R_{b,0}^{}, R_{b,0}'$ are two copies of a component of a taut Seifert surface for $L$ that are disjoint from $R_a$ and are not isotopic to any component of it.
Then $R_{b,0}^{},R_{b,0}'$ are isotopic by an isotopy that does not meet $R_a$.
\end{lemma}
\begin{proof}
By a small isotopy disjoint from $R_a$, we may ensure that $R_{b,0}^{},R_{b,0}'$ are transverse.
Since they are isotopic, there is a product region $N$ between them. 
If $N$ meets $R_a$, it contains a whole component of $R_a$, which is then isotopic to each of the horizontal boundary components of $N$. This contradicts that $R_{b,0}^{}$ and $R_{b,0}'$ are not isotopic to any component of $R_a$. Thus $N$ is disjoint from $R_a$.
If $R_{b,0}^{}\cap R_{b,0}'\neq \emptyset$ then the isotopy defined by $N$ reduces $|R_{b,0}^{}\cap R_{b,0}'|$. If $R_{b,0}^{}\cap R_{b,0}'=\emptyset$ then the isotopy makes $R_{b,0}^{}$ and $R_{b,0}'$ coincide.
\end{proof}

\begin{lemma}\label{makedisjointlemma}
Let $R_a,R_b$ be adjacent vertices of $\ms(L)$. Then $R_a,R_b$ can be isotoped so they are disjoint and realise their adjacency. 

Suppose there are components $R_{a,0}$ of $R_a$ and $R_{b,0}$ of $R_b$ that can be made to coincide, so there is a product region between these components. 
The side of $R_a$ on which this product region lies is determined by the choice of $R_a,R_b$.
\end{lemma}
\begin{proof}
We regard $R_b$ as fixed, and isotope $R_a$.
Isotope $R_a$ to realise the adjacency. Suppose they are not disjoint, and pick a component $R_{a,0}$ of $R_a$ that is not disjoint from $R_b$. Because the surfaces realise their adjacency, $R_{a,0}$ cannot cross $R_b$. Therefore $R_{a,0}$ can be pushed off $R_b$ by a small isotopy. If the two components do not coincide, there is no choice as to which direction to push $R_{a,0}$, and it is clear that the condition in Lemma \ref{arcintersectionlemma} continues to hold.

If $R_{a,0}$ coincides with a component of $R_b$, it is possible to push it off in either direction, creating a product region. We will see that the choice of direction is forced upon us by wanting the condition in Lemma \ref{arcintersectionlemma} to continue to hold.

As $R_a\neq R_b$, at least one component $R_{a,1}$ of $R_a$ does not coincide with any component of $R_b$. Fix a point $x_0$ of $R_{a,1}$, and
choose a product neighbourhood $\nhd(R_a)=R_a\times[-1,1]$ 
such that $R_a=R_a\times\{0\}$. For each point $y$ of $R_{a,0}$, choose a path $\rho$ from $(x_0,1)$ to $(y,-1)$ that is disjoint from $R_a$ and transverse to $R_b$. 
Since $\rho$ is contained in $M\setminus R_a$, it has algebraic intersection 0 or 1 with $R_b$. 
Suppose that this number is not both well-defined and constant on $R_{a,0}$. That is, suppose there are such points $y_0,y_1$ and paths $\rho_0,\rho_1$ such that one gives the value 0 while the other gives the value 1. 
Let $\rho$ be a path from $y_0$ to $y_1$ in $R_{a,0}\times\{-1\}$.
Then $\rho\cup\rho_0\cup\rho_1$ forms a closed curve that has intersection 0 with $R_a$ but intersection 1 with $R_b$. This is not possible. 
We therefore have a well-defined value for the algebraic intersection of $R_b$ and a path as described. 
If this value is 1, push $R_{a,0}$ downwards, and otherwise push it upwards. 
Then, for $y\in R_{a,0}$, any path from $(x,1)$ to $(y,-1)$ that is disjoint from $R_a$ has intersection 1 with $R_b$.

Now pick points $x,y$ and path $\rho$ as in the condition in Lemma \ref{arcintersectionlemma}. Isotope $\rho$ so that it decomposes into three paths disjoint from $R_a$, one from $(x,1)$ to $(x_0,-1)$, one from $(x_0,-1)$ to $(x_0,1)$ and one from $(x_0,1)$ to $(y,-1)$. The outer two paths have intersection $1$ with $R_b$ while the middle one has intersection $-1$ with $R_b$. Thus $\rho$ has intersection $1$ with $R_b$.
\end{proof}

\section{An ordering on the vertices of $\ms(L_i)$}\label{orderingsection}

In \cite{przytycki-2010}, a partial ordering $<_{S}$ is defined on the vertices of $\ms(L)$ for a link $L$, relative to a fixed vertex $S$. This ordering only compares adjacent vertices.

\begin{definition}[\cite{przytycki-2010} Section 5]\label{orderingdefn}
Let $R,R',S$ be vertices of $\ms(L)$ with $R,R'$ adjacent. Isotope the surfaces so that $R,R'$ are almost transverse to and have simplified intersection with $S$, and so that $R,R'$ are almost disjoint with simplified intersection.
Set $M=\Sphere\setminus\nhd(L)$. Let $\tilde{M}$ denote the infinite cyclic cover of $M$, and let $\tau$ be the generating covering transformation (in the positive direction). 
Let $V_{S}$ be a lift of $M\setminus S$.

Let $V_{R}$ be the lift of $M\setminus R$ such that $V_{R}\cap V_{S}\neq\emptyset$ but $V_{R}\cap \tau(V_{S})=\emptyset$.
Finally, let $V_{R'}$ be the lift of $M\setminus R'$ such that $V_{R'}\cap V_{R}\neq\emptyset$ but $V_{R'}\cap\tau(V_{R})=\emptyset$. 
See Figure \ref{orderingpic1}.

Then $R' <_{S} R$ if $V_{R'}\cap V_S\neq\emptyset$.
\begin{figure}[htbp]
\centering
\psset{xunit=.5pt,yunit=.5pt,runit=.5pt}
\begin{pspicture}(270,330)
{
\pscustom[linewidth=1,linecolor=black]
{
\newpath
\moveto(50,320)
\lineto(50,10.000003)
\moveto(220,320)
\lineto(220,10.000003)
\moveto(50,300)
\lineto(220,300)
\moveto(50,210.000003)
\lineto(220,210.000003)
\moveto(50,120.000003)
\lineto(220,120.000003)
\moveto(220,130.000003)
\curveto(170,130.000003)(180,240.000003)(160,210.000003)
\curveto(136.46606,174.699103)(127.33567,244.579753)(120,229.999993)
\curveto(101.46842,193.168213)(100,140.000003)(50,140.000003)
\moveto(50,170.000003)
\curveto(90,190.000003)(101.15074,311.397455)(150,270.000003)
\curveto(197.74588,229.537613)(159.32956,163.708633)(220,170.000003)
\moveto(230,210.000003)
\lineto(230,299.999992)
\moveto(50,80.000003)
\curveto(90,100.000003)(101.15074,221.397453)(150,180.000003)
\curveto(197.74588,139.537613)(159.32956,73.708633)(220,80.000003)
\moveto(40,300)
\lineto(40,210.000003)
\moveto(40,170.000003)
\lineto(40,80.000003)
\moveto(220,40.000013)
\curveto(170,40.000013)(180,150.000003)(160,120.000003)
\curveto(136.46606,84.699103)(127.33567,154.579753)(120,139.999993)
\curveto(101.46842,103.168213)(100,50.000003)(50,50.000003)
\moveto(50,30.000003)
\lineto(220,30.000003)
\moveto(230,130.000003)
\lineto(230,40.000003)
}
}
{
\pscustom[linewidth=1,linecolor=black,fillstyle=solid,fillcolor=black]
{
\newpath
\moveto(230,289.999992)
\lineto(234,285.999992)
\lineto(230,299.999992)
\lineto(226,285.999992)
\lineto(230,289.999992)
\closepath
}
}
{
\put(235,250){$\tau$}
\put(10,250){$V_S$}
\put(10,130){$V_R$}
\put(235,90){$V_{R'}$}
}
\end{pspicture}
\caption{\label{orderingpic1}}
\end{figure}
\end{definition}



\begin{lemma}[\cite{przytycki-2010} Lemma 5.3]\label{orderinglemma1}
Let $R, R'$ be adjacent vertices, and $S$ any vertex. Then $R' <_{S}R$ or $R <_{S}R'$.
\end{lemma}

\begin{lemma}[\cite{przytycki-2010} Lemma 5.4]\label{orderdistancelemma}
If $\dist_{\ms(L)}(R',S)<\dist_{\ms(L)}(R,S)$ then ${R<_S R'}$.
\end{lemma}

\begin{lemma}[\cite{przytycki-2010} Lemma 5.5]\label{orderinglemma2}
There are no $R_1,\cdots,R_k$, for $k\geq 2$, with $R_1 <_{S} R_2 <_{S}\cdots<_{S}R_k <_{S}R_1$.
\end{lemma}

Now choose $L,L_1,L_2$ as in the statement of Theorem \ref{connectsumthm}. These will remain fixed for the remainder of this paper.

\begin{definition}
For $i=1,2$, let $K_i$ be the link component of $L_i$ along which the connected sum is performed.
Let $T_0$ be a fixed copy of $\sphere$ that divides $L$ into $L_1$ and $L_2$, and choose a product neighbourhood $T_0\times[1,2]$ such that $T_0=T_0\times\left\{\frac{3}{2}\right\}$ and $T_0\times\{i\}$ lies on the same side of $T_0$ as $L_i$ for $i=1,2$. We further require that both arcs of $L\cap(T_0\times[1,2])$ are of the form $\{x\}\times[1,2]$ for some $x\in T_0$.

Next, choose a regular neighbourhood $\nhd(L)$ of $L$, and let $M=\Sphere\setminus\nhd(L)$. In addition, let $M_0$ denote $M\cap (T_0\times[1,2])$ and for $i=1,2$ denote by $M_i$ the component of $M\setminus(T_0\times(1,2))$ that meets $L_i$. See Figure \ref{orderingpic2}a.
\end{definition}
\begin{figure}[htbp]
\centering
(a)
\psset{xunit=.30pt,yunit=.30pt,runit=.30pt}
\begin{pspicture}(521,511)
{
\newgray{lightgrey}{.8}
\newgray{lightishgrey}{.7}
\newgray{grey}{.6}
\newgray{midgrey}{.4}
\newgray{darkgrey}{.3}
}
{
\pscustom[linewidth=1.5,linecolor=black]
{
\newpath
\moveto(275.5,505.500051)
\lineto(275.5,250.500093)
\lineto(185.5,70.500003)
\moveto(410.5,505.500031)
\lineto(410.5,250.500073)
\lineto(320.5,70.500003)
\moveto(140.5,505.500011)
\lineto(140.5,250.500053)
\lineto(50.5,70.500003)
}
}
{
\pscustom[linestyle=none,fillstyle=solid,fillcolor=white]
{
\newpath
\moveto(100.5,320.500003)
\lineto(470.5,320.500003)
\lineto(470.5,290.500003)
\lineto(100.5,290.500003)
\lineto(100.5,320.500003)
\closepath
}
}
{
\pscustom[linestyle=none,fillstyle=solid,fillcolor=lightgrey]
{
\newpath
\moveto(465.5,299.500003)
\lineto(405.5,179.500003)
\lineto(35.5,179.500003)
\lineto(95.5,299.500003)
\lineto(465.5,299.500003)
\closepath
}
}
{
\pscustom[linewidth=1.5,linecolor=black]
{
\newpath
\moveto(100.5,290.500003)
\lineto(470.5,290.500003)
\moveto(100.5,320.500003)
\lineto(470.5,320.500003)
\moveto(195.5,180.000003)
\lineto(255.5,300.000003)
}
}
{
\pscustom[linewidth=1.5,linecolor=black,linestyle=dashed,dash=6 6]
{
\newpath
\moveto(60.5,180.000003)
\lineto(120.5,300.000003)
\moveto(330.5,180.000003)
\lineto(390.5,300.000003)
}
}
{
\pscustom[linestyle=none,fillstyle=solid,fillcolor=white]
{
\newpath
\moveto(30.500002,200.500003)
\lineto(410.5,200.500003)
\lineto(410.5,170.500003)
\lineto(30.500002,170.500003)
\lineto(30.500002,200.500003)
\closepath
}
}
{
\pscustom[linewidth=1.5,linecolor=black,fillstyle=solid,fillcolor=white]
{
\newpath
\moveto(414.97441006,185.50000038)
\curveto(414.97441006,177.21572914)(412.97114844,170.50000038)(410.5,170.50000038)
\curveto(408.02885156,170.50000038)(406.02558994,177.21572914)(406.02558994,185.50000038)
\curveto(406.02558994,193.78427163)(408.02885156,200.50000038)(410.5,200.50000038)
\curveto(412.97114844,200.50000038)(414.97441006,193.78427163)(414.97441006,185.50000038)
\closepath
\moveto(474.44882006,305.50000038)
\curveto(474.44882006,297.21572914)(472.44555844,290.50000038)(469.97441,290.50000038)
\curveto(467.50326156,290.50000038)(465.49999994,297.21572914)(465.49999994,305.50000038)
\curveto(465.49999994,313.78427163)(467.50326156,320.50000038)(469.97441,320.50000038)
\curveto(472.44555844,320.50000038)(474.44882006,313.78427163)(474.44882006,305.50000038)
\closepath
}
}
{
\pscustom[linewidth=1.5,linecolor=black]
{
\newpath
\moveto(40.5,170.500003)
\lineto(410.5,170.500003)
\moveto(40.5,200.500003)
\lineto(410.5,200.500003)
\moveto(416.02559,185.500003)
\lineto(406.02559,185.500003)
\moveto(40.5,40.500003)
\lineto(310.5,40.500003)
\moveto(100.47695844,290.59491251)
\curveto(98.02144613,289.6644552)(95.8058615,295.58340566)(95.52831166,303.81525839)
\curveto(95.25076183,312.04711111)(97.01634926,319.47463094)(99.47186156,320.40508825)
\curveto(99.80583053,320.53163776)(100.14298944,320.53163776)(100.4769584,320.40508826)
\moveto(40.47695844,170.59491251)
\curveto(38.02144613,169.6644552)(35.8058615,175.58340566)(35.52831166,183.81525839)
\curveto(35.25076183,192.04711111)(37.01634926,199.47463094)(39.47186156,200.40508825)
\curveto(39.80583053,200.53163776)(40.14298944,200.53163776)(40.4769584,200.40508826)
\moveto(475.5,305.500003)
\lineto(465.5,305.500003)
\moveto(140.5,505.500011)
\lineto(50.5,325.500023)
\lineto(50.5,70.500003)
\moveto(275.5,505.500051)
\lineto(185.5,325.500063)
\lineto(185.5,70.500003)
\moveto(410.5,505.500031)
\lineto(320.5,325.500043)
\lineto(320.5,70.500003)
\moveto(199.9626044,170.50005259)
\curveto(197.49146457,170.5219104)(195.49349547,177.25533502)(195.50001552,185.53957743)
\curveto(195.50545209,192.44716908)(196.91732903,198.44922355)(198.91977467,200.07737462)
\moveto(334.9626044,170.50005259)
\curveto(332.49146457,170.5219104)(330.49349547,177.25533502)(330.50001552,185.53957743)
\curveto(330.50622683,193.4315474)(332.33552781,199.95758044)(334.68472432,200.46853018)
\moveto(64.9626044,170.50005259)
\curveto(62.49146457,170.5219104)(60.49349547,177.25533502)(60.50001552,185.53957743)
\curveto(60.50545209,192.44716908)(61.91732903,198.44922355)(63.91977467,200.07737462)
\moveto(259.4370144,290.50005259)
\curveto(256.96587457,290.5219104)(254.96790547,297.25533502)(254.97442552,305.53957743)
\curveto(254.97986209,312.44716908)(256.39173903,318.44922355)(258.39418467,320.07737462)
\moveto(394.4370144,290.50005259)
\curveto(391.96587457,290.5219104)(389.96790547,297.25533502)(389.97442552,305.53957743)
\curveto(389.98063683,313.4315474)(391.80993781,319.95758044)(394.15913432,320.46853018)
\moveto(124.4370144,290.50005259)
\curveto(121.96587457,290.5219104)(119.96790547,297.25533502)(119.97442552,305.53957743)
\curveto(119.97986209,312.44716908)(121.39173903,318.44922355)(123.39418467,320.07737462)
}
}
{
\pscustom[linestyle=none,fillstyle=solid,fillcolor=black]
{
\newpath
\moveto(425.22558999,185.500003)
\curveto(425.22558999,185.94183081)(425.58376219,186.30000301)(426.02559,186.30000301)
\curveto(426.46741781,186.30000301)(426.82559001,185.94183081)(426.82559001,185.500003)
\curveto(426.82559001,185.05817519)(426.46741781,184.70000299)(426.02559,184.70000299)
\curveto(425.58376219,184.70000299)(425.22558999,185.05817519)(425.22558999,185.500003)
\closepath
\moveto(421.72558999,185.500003)
\curveto(421.72558999,185.94183081)(422.08376219,186.30000301)(422.52559,186.30000301)
\curveto(422.96741781,186.30000301)(423.32559001,185.94183081)(423.32559001,185.500003)
\curveto(423.32559001,185.05817519)(422.96741781,184.70000299)(422.52559,184.70000299)
\curveto(422.08376219,184.70000299)(421.72558999,185.05817519)(421.72558999,185.500003)
\closepath
\moveto(418.22558999,185.500003)
\curveto(418.22558999,185.94183081)(418.58376219,186.30000301)(419.02559,186.30000301)
\curveto(419.46741781,186.30000301)(419.82559001,185.94183081)(419.82559001,185.500003)
\curveto(419.82559001,185.05817519)(419.46741781,184.70000299)(419.02559,184.70000299)
\curveto(418.58376219,184.70000299)(418.22558999,185.05817519)(418.22558999,185.500003)
\closepath
\moveto(484.69999999,305.500003)
\curveto(484.69999999,305.94183081)(485.05817219,306.30000301)(485.5,306.30000301)
\curveto(485.94182781,306.30000301)(486.30000001,305.94183081)(486.30000001,305.500003)
\curveto(486.30000001,305.05817519)(485.94182781,304.70000299)(485.5,304.70000299)
\curveto(485.05817219,304.70000299)(484.69999999,305.05817519)(484.69999999,305.500003)
\closepath
\moveto(481.19999999,305.500003)
\curveto(481.19999999,305.94183081)(481.55817219,306.30000301)(482,306.30000301)
\curveto(482.44182781,306.30000301)(482.80000001,305.94183081)(482.80000001,305.500003)
\curveto(482.80000001,305.05817519)(482.44182781,304.70000299)(482,304.70000299)
\curveto(481.55817219,304.70000299)(481.19999999,305.05817519)(481.19999999,305.500003)
\closepath
\moveto(477.69999999,305.500003)
\curveto(477.69999999,305.94183081)(478.05817219,306.30000301)(478.5,306.30000301)
\curveto(478.94182781,306.30000301)(479.30000001,305.94183081)(479.30000001,305.500003)
\curveto(479.30000001,305.05817519)(478.94182781,304.70000299)(478.5,304.70000299)
\curveto(478.05817219,304.70000299)(477.69999999,305.05817519)(477.69999999,305.500003)
\closepath
}
}
{
\put(490,290){$L$}
\put(440,230){$S_0^*$}
\put(240,250){$\sigma^*$}
\put(200,500){\small$T_0\times\{\frac{3}{2}\}$}
\put(10,500){\small$T_0\times\{1\}$}
\put(370,500){\small$T_0\times\{2\}$}
\put(160,10){$M_0$}
\put(0,80){$M_1$}
\put(340,80){$M_2$}
\put(390,140){$\nhd(L)$}
}
\end{pspicture}
(b)
\psset{xunit=.30pt,yunit=.30pt,runit=.30pt}
\begin{pspicture}(521,511)
{
\newgray{lightgrey}{.8}
\newgray{lightishgrey}{.7}
\newgray{grey}{.6}
\newgray{midgrey}{.4}
\newgray{darkgrey}{.3}
}
{
\pscustom[linewidth=1.5,linecolor=black]
{
\newpath
\moveto(275.5,505.50004)
\lineto(275.5,250.50007)
\lineto(185.5,70.49998)
\moveto(410.5,505.50002)
\lineto(410.5,250.50005)
\lineto(320.5,70.49998)
\moveto(140.5,505.5)
\lineto(140.5,250.50003)
\lineto(50.5,70.49998)
}
}
{
\pscustom[linestyle=none,fillstyle=solid,fillcolor=white]
{
\newpath
\moveto(128.5,320.5)
\lineto(60.5,185.5)
\lineto(325.5,185.5)
\lineto(393.5,320.5)
\lineto(128.5,320.5)
\closepath
}
}
{
\pscustom[linestyle=none,fillstyle=solid,fillcolor=white]
{
\newpath
\moveto(110.5,320.50003)
\lineto(480.5,320.50003)
\lineto(480.5,290.50003)
\lineto(110.5,290.50003)
\lineto(110.5,320.50003)
\closepath
}
}
{
\pscustom[linestyle=none,fillstyle=solid,fillcolor=lightgrey]
{
\newpath
\moveto(466.5,299.50003)
\lineto(406.5,179.5)
\lineto(331.5,179.49997)
\lineto(391.5,299.5)
\lineto(466.5,299.50003)
\closepath
}
}
{
\pscustom[linewidth=1.5,linecolor=black]
{
\newpath
\moveto(395.5,290.5)
\lineto(470.5,290.50003)
\moveto(123.5,320.5)
\lineto(470.5,320.50003)
\moveto(335.5,170.49997)
\lineto(395.5,290.5)
}
}
{
\pscustom[linewidth=1.5,linecolor=black,linestyle=dashed,dash=6 6]
{
\newpath
\moveto(330.5,177.49997)
\lineto(390.5,297.5)
}
}
{
\pscustom[linestyle=none,fillstyle=solid,fillcolor=white]
{
\newpath
\moveto(30.5,200.5)
\lineto(410.5,200.5)
\lineto(410.5,170.5)
\lineto(30.5,170.5)
\lineto(30.5,200.5)
\closepath
}
}
{
\pscustom[linewidth=1,linecolor=midgrey]
{
\newpath
\moveto(405.5,185.5)
\lineto(200.5,185.5)
\lineto(260.5,305.50003)
\lineto(465.5,305.5)
}
}
{
\pscustom[linewidth=1.5,linecolor=black,fillstyle=solid,fillcolor=white]
{
\newpath
\moveto(414.97441006,185.49999738)
\curveto(414.97441006,177.21572614)(412.97114844,170.49999738)(410.5,170.49999738)
\curveto(408.02885156,170.49999738)(406.02558994,177.21572614)(406.02558994,185.49999738)
\curveto(406.02558994,193.78426863)(408.02885156,200.49999738)(410.5,200.49999738)
\curveto(412.97114844,200.49999738)(414.97441006,193.78426863)(414.97441006,185.49999738)
\closepath
\moveto(474.44882006,305.50001738)
\curveto(474.44882006,297.21574614)(472.44555844,290.50001738)(469.97441,290.50001738)
\curveto(467.50326156,290.50001738)(465.49999994,297.21574614)(465.49999994,305.50001738)
\curveto(465.49999994,313.78428863)(467.50326156,320.50001738)(469.97441,320.50001738)
\curveto(472.44555844,320.50001738)(474.44882006,313.78428863)(474.44882006,305.50001738)
\closepath
}
}
{
\newrgbcolor{curcolor}{0 0 0}
\pscustom[linewidth=1,linecolor=black]
{
\newpath
\moveto(65.5,170.5)
\lineto(410.5,170.5)
\moveto(333.5,200.5)
\lineto(410.5,200.5)
\moveto(416.02559,185.49997)
\lineto(406.02559,185.49997)
\moveto(475.5,305.50003)
\lineto(465.5,305.50003)
\moveto(140.5,505.5)
\lineto(50.5,325.5)
\lineto(50.5,70.49998)
\moveto(333.5,200.49997)
\lineto(393.5,320.5)
\moveto(63.5,200.49997)
\lineto(123.5,320.5)
\moveto(395.47695844,290.59490951)
\curveto(393.02144613,289.6644522)(390.8058615,295.58340266)(390.52831166,303.81525539)
\curveto(390.25076183,312.04710811)(392.01634926,319.47462794)(394.47186156,320.40508525)
\curveto(394.80583053,320.53163476)(395.14298944,320.53163476)(395.4769584,320.40508526)
\moveto(275.5,505.50004)
\lineto(185.5,325.50004)
\lineto(185.5,70.49998)
\moveto(410.5,505.50002)
\lineto(320.5,325.50002)
\lineto(320.5,70.49998)
\moveto(198.5,200.5)
\lineto(258.5,320.50003)
\moveto(199.9626044,170.50004959)
\curveto(197.49146457,170.5219074)(195.49349547,177.25533202)(195.50001552,185.53957443)
\curveto(195.50545209,192.44716608)(196.91732903,198.44922055)(198.91977467,200.07737162)
\moveto(334.9626044,170.50004959)
\curveto(332.49146457,170.5219074)(330.49349547,177.25533202)(330.50001552,185.53957443)
\curveto(330.50622683,193.4315444)(332.33552781,199.95757744)(334.68472432,200.46852718)
\moveto(64.9626044,170.50004959)
\curveto(62.49146457,170.5219074)(60.49349547,177.25533202)(60.50001552,185.53957443)
\curveto(60.50545209,192.44716608)(61.91732903,198.44922055)(63.91977467,200.07737162)
}
}
{
\pscustom[linestyle=none,fillstyle=solid,fillcolor=black]
{
\newpath
\moveto(425.22558999,185.49997)
\curveto(425.22558999,185.94179781)(425.58376219,186.29997001)(426.02559,186.29997001)
\curveto(426.46741781,186.29997001)(426.82559001,185.94179781)(426.82559001,185.49997)
\curveto(426.82559001,185.05814219)(426.46741781,184.69996999)(426.02559,184.69996999)
\curveto(425.58376219,184.69996999)(425.22558999,185.05814219)(425.22558999,185.49997)
\closepath
\moveto(421.72558999,185.49997)
\curveto(421.72558999,185.94179781)(422.08376219,186.29997001)(422.52559,186.29997001)
\curveto(422.96741781,186.29997001)(423.32559001,185.94179781)(423.32559001,185.49997)
\curveto(423.32559001,185.05814219)(422.96741781,184.69996999)(422.52559,184.69996999)
\curveto(422.08376219,184.69996999)(421.72558999,185.05814219)(421.72558999,185.49997)
\closepath
\moveto(418.22558999,185.49997)
\curveto(418.22558999,185.94179781)(418.58376219,186.29997001)(419.02559,186.29997001)
\curveto(419.46741781,186.29997001)(419.82559001,185.94179781)(419.82559001,185.49997)
\curveto(419.82559001,185.05814219)(419.46741781,184.69996999)(419.02559,184.69996999)
\curveto(418.58376219,184.69996999)(418.22558999,185.05814219)(418.22558999,185.49997)
\closepath
\moveto(484.69999999,305.50003)
\curveto(484.69999999,305.94185781)(485.05817219,306.30003001)(485.5,306.30003001)
\curveto(485.94182781,306.30003001)(486.30000001,305.94185781)(486.30000001,305.50003)
\curveto(486.30000001,305.05820219)(485.94182781,304.70002999)(485.5,304.70002999)
\curveto(485.05817219,304.70002999)(484.69999999,305.05820219)(484.69999999,305.50003)
\closepath
\moveto(481.19999999,305.50003)
\curveto(481.19999999,305.94185781)(481.55817219,306.30003001)(482,306.30003001)
\curveto(482.44182781,306.30003001)(482.80000001,305.94185781)(482.80000001,305.50003)
\curveto(482.80000001,305.05820219)(482.44182781,304.70002999)(482,304.70002999)
\curveto(481.55817219,304.70002999)(481.19999999,305.05820219)(481.19999999,305.50003)
\closepath
\moveto(477.69999999,305.50003)
\curveto(477.69999999,305.94185781)(478.05817219,306.30003001)(478.5,306.30003001)
\curveto(478.94182781,306.30003001)(479.30000001,305.94185781)(479.30000001,305.50003)
\curveto(479.30000001,305.05820219)(478.94182781,304.70002999)(478.5,304.70002999)
\curveto(478.05817219,304.70002999)(477.69999999,305.05820219)(477.69999999,305.50003)
\closepath
}
}
{
\put(490,290){$L_2$}
\put(450,240){$S_2^*$}
\put(390,140){$\nhd(L_2)$}
\put(240,250){$\sigma^*$}
}
\end{pspicture}
\caption{\label{orderingpic2}}
\end{figure}

\begin{definition}
Choose a taut Seifert surface $S_0$ for $L$. We will use $S_0$ as a basepoint for $\ms(L)$.  Isotope $S_0$ to have minimal intersection with $T_0$. Then $S_0\cap T_0$ is a single arc $\sigma^*$. Further ensure that $S_0\cap M_0=\sigma^*\times[1,2]$. Let $S_0^*$ denote this copy of $S_0$, considered as a fixed surface rather than up to isotopy. Again, see Figure \ref{orderingpic2}a.

The link made up of the part of $L$ on the $M_1$ side of $T_0$ together with the arc $\sigma^*$ is $L_1$, and $M_1$ is homeomorphic to $\Sphere\setminus\nhd(L_1)$. 
The same is true for $L_2$.

The sphere $T_0$ divides $S_0^*$ into Seifert surfaces $S_1,S_2$ for $L_1,L_2$ respectively. Since $S_0$ is taut, so are $S_1,S_2$. Let $S_1^*, S_2^*$ be these surfaces, again considered as fixed.
Define curves $\lambda^*,\lambda_1^*,\lambda_2^*$ on $\partial M,\partial M_1,\partial M_2$ respectively, also seen as fixed, by $\lambda^*=\partial S_0^*$ and $\lambda_i^*=\partial S_i^*$ for $i=1,2$. By an appropriate choice of $\nhd(L_i)$ for $i=1,2$ we may ensure that $(\sigma^*\times\{i\})\cap M\subset\lambda_i^*$. See Figure \ref{orderingpic2}b.
\end{definition}

\begin{definition}
Define $\leq$ on $\V(\ms(L_1))$ by ${R\leq R'} \Leftrightarrow {(R<_{S_1}R'} \textrm{ or }{R=R')}$. 
Similarly, define $\leq$ on $\V(\ms(L_2))$ by ${R\leq R'} \Leftrightarrow {(R<_{S_2}R'} \textrm{ or } {R=R')}$. 
\end{definition}

\begin{corollary}
The pairs $(\ms(L_1),\leq),(\ms(L_2),\leq)$ are ordered simplicial complexes.
\end{corollary}
\begin{proof}
Lemma \ref{orderinglemma1} gives (P2)$'$, and Lemma \ref{orderinglemma2} gives (P1). Together they give (P3).
\end{proof}

These two orderings, together with that on $\V(\mathcal{Z})$, allow us to apply Theorem \ref{productcomplextheorem}. We will use this to give a triangulation of $|\ms(L_1)|\times|\ms(L_2)|\times\mathbb{R}$ that agrees with the triangulation of $\ms(L)$ in a natural way.

\section{Choosing representatives of isotopy classes}\label{repssection}

Given taut Seifert surfaces $R_i$ of $L_i$ for $i=1,2$, we can isotope the surfaces so that $\partial R_i=\lambda_i^*$. Having done so, we can glue each of the two surfaces to the rectangle $\sigma^*\times[1,2]$ to form a taut Seifert surface $R$ of $L$ with $\partial R=\lambda^*$.

An isotopy of any Seifert surface $R$ for $L$ can be split into an isotopy that fixes $\partial R$ and an isotopy supported in a neighbourhood of $\partial M$. Thus an isotopy class relative to the boundary corresponds to a fixed element of the isotopy class together with a winding number for each boundary component.
To measure the winding numbers, we need to decide what it means to have winding number 0. We use $S^*_i$ as a basepoint for $\ms(L_i)$, setting this to have winding number 0 at each boundary component. We want to define what it means for another surface to have winding number 0. In practice we will only be concerned with the winding number at $K_i$, but it is convenient to choose surfaces fixed at every boundary component.

Thus our present aim is to find a fixed representative $R^*$ for each vertex $R$ of $\ms(L_i)$. We choose these such that $\partial R^*=\lambda_i^*$. We also want these representatives to interact well with regard to the ordering $\leq$.

\begin{definition}
Let $R,R'$ be $\partial$--almost disjoint taut Seifert surfaces for $L_i$. Pick a component $K'$ of $L_i$, and consider $R,R'$ near $K'$. It may be that the components that meet $K'$ coincide. If not, one of the two surfaces lies `above' the other, where this is measured in the positive direction around $K'$. Write $R\leq_{K'} R'$ if either $R'$ lies above $R$ or the two coincide.
\end{definition}

\begin{definition}
Define a relation $\leq_{\partial}$ on isotopy classes of taut Seifert surfaces relative to the boundary by $R\leq_{\partial} R'$ if there are representatives $R_b$ of $R$ and $R'_a$ of $R'$ such that $R_b,R'_a$ are $\partial$--almost disjoint and $R_b\leq_{K'}R'_a$ for each component $K'$ of $L_i$.
\end{definition}

\begin{lemma}
The relation $\leq_{\partial}$ is antisymmetric.
\end{lemma}
\begin{proof}
Suppose otherwise. Choose $R,R'$ with $R\leq_{\partial}R'\leq_{\partial}R$ and $R\neq R'$. Consider $R'$ as fixed, and choose representatives $R_a,R_b$ of $R$ such that $R_a$ shows that $R'\leq_{\partial} R$ and $R_b$ shows that $R\leq_{\partial}R'$.    
Note that $\partial R'=\partial R_a=\partial R_b=\lambda^*_i$.

The surface $R_a$ might be disconnected. However, at least one component of $R_a$ is not isotopic to any component of $R'$. Remove all other components of $R_a$, and the corresponding ones of $R_b$. Also remove all components of $R'$ that are disjoint from the new $R_a$.

As $R_a,R_b$ only meet a neighbourhood of $R'$ along their boundaries, where we know how they are positioned, we see that $R_a,R_b$ can be put into general position by an isotopy away from a neighbourhood of $R'$. Choose this isotopy to minimise $|R_a\cap R_b|$.

Consider the product region $N$ between $R_a$ and $R_b$ given by Corollary \ref{isotopyrelbdycor}. The isotopy of $R_a$ defined by $N$ does not move $\partial R_a$, so our positioning of $R_a,R_b$ means that $N$ must meet $R'$. Let $R'_0$ be a component of $R'$ that meets $N$. Then $R'_0\subset N$. Since $R'_0$ is $\partial$--almost disjoint from $R_a$ and from $R_b$, Proposition \ref{surfaceinproductprop} gives that $R'_0$ is isotopic to $R_a$, which is a contradiction.
\end{proof}

\begin{proposition}\label{representativesprop}
It is possible to choose a representative $R^*$ for each vertex $R$ of $\ms(L_i)$ such that the following conditions hold for any pair $(R_a,R_b)$ of adjacent vertices of $\ms(L_i)$.
\begin{itemize}
 \item $\partial R_a^*=\partial R_b^*=\lambda_i^*$.
 \item There are $\partial$--almost disjoint copies $R'_j$ of $R_j$, for $j=a,b$, that are isotopic to $R_j^*$ via isotopies fixing $\partial R'_j$ and that demonstrate their adjacency.
 \item $R_b^*\leq_{\partial} R_a^*$ (equivalently, $R'_b\leq_{\partial} R'_a$) if and only if $R_b\leq R_a$.
\end{itemize}
\end{proposition}
\begin{proof}
Without loss of generality, $i=1$. Let $\tilde{M}_1$ be the infinite cyclic cover of $M_1$, with covering transformation $\tau_1$.
As in Proposition \ref{msdistanceprop}, construct a lift $V_n$ of $M_1\setminus S_1^*$ for $n\in\mathbb{Z}$. Let $\tilde{S}_1$ be the lift of $S_1^*$ that lies between $V_0$ and $V_1$.

We have already chosen the representative $S_1^*$ for $S_1$. Let $R$ be a vertex of $\ms(L_1)$ other than $S_1$. Isotope $R$ to be almost transverse to and have simplified intersection with $S_1^*$. Let $V_R$ be the lift of $M_1\setminus R$ such that $V_R\cap V_0\neq\emptyset$ but $V_R\cap V_1=\emptyset$, and let $\tilde{R}$ be the lift of $R$ that lies between $V_R$ and $\tau_1(V_R)$. From Proposition \ref{realisedistanceprop} we see that $R$ minimises $\max\{n:V_{R}\cap V_n\neq\emptyset\}-\min\{n:V_{R}\cap V_n\neq\emptyset\}$. By Proposition \ref{msdistanceprop} this means that $\{n:V_R\cap V_n\neq\emptyset\}=\{0,-1,\ldots,-\dist_{\ms(L_1)}(S_1,R)\}$.

By an isotopy close to the boundary, move $\partial R$ around $\partial\!\nhd(L_1)$ until $\partial \tilde{R}$ coincides with $\partial \tilde{S}_1$. Note that this does not change $\{n:V_R\cap{} V_n\neq\emptyset\}$. Take the resulting surface to be $R^*$.

Our first aim is to show that the choice of $R^*$ is well-defined up to isotopy relative to the boundary. That is, we wish to check that the choice of winding number is independent of the choice of isotopy made when constructing $R^*$.
Let $R'$ be any other copy of $R$. Construct $(R')^*$ as described. 
Then $\partial \tilde{R}=\partial \tilde{R}'=\partial \tilde{S}_1$. 
Note that we do not know how $R$ and $R'$ intersect.
However, $R$ is isotopic to $R'$ in $M_1$. This isotopy lifts to an isotopy in $\tilde{M}_1$ from $\tilde{R}$ to a lift of $R'$. This lift is $\tau_1^m(\tilde{R}')$ for some $m\in\mathbb{Z}$. The isotopy also takes $V_{R}$ to $\tau_1^m(V_{R'})$.

Suppose $m\neq 0$. Without loss of generality, $m>0$. Then the submanifold $\tau_1^m(V_{R'}\cup\tilde{R}'\cup\tau_1^{-1}(\tilde{R}'))$ of $\tilde{M}_1$ has interior disjoint from the submanifold $\tau_1^{-k}(V_0\cup \tilde{S}_1\cup\tau_1^{-1}(\tilde{S}_1))$, where $k=\dist_{\ms(L_1)}(S_1,R)$. Thus by Proposition \ref{disjointinteriorsprop} we see that $V_R$ is disjoint from $V_{-k}$. This contradicts that $\{n:V_R\cap V_n\neq\emptyset\}=\{0,-1,\ldots,-k\}$.
Hence $m=0$. We can therefore modify the isotopy from $R$ to $R'$ near the boundary to keep $\partial R$ fixed throughout. Thus $R^*$ and $(R')^*$ are isotopic relative to the boundary, as required.

It remains to show that our chosen representatives have the required properties. Let $R_a,R_b$ be adjacent vertices of $\ms(L_1)$ with $R_b\leq R_a$. Position them relative to $S_1^*$ and each other as in Lemma \ref{maketransverselemma}. Then by Proposition \ref{realisedistanceprop} we know that $R_a,R_b$ realise their adjacency, and so are almost disjoint. 

Now consider the lifts used to demonstrate that $R_b<_{S_1} R_a$, as in Definition \ref{orderingdefn}. Taking $V_0$ as the lift of $M_1\setminus S_1^*$ we see that $V_{R_a}$ is the required lift of $M_1\setminus R_a$. Let $V_{R_b}'$ be the required lift of $M_1\setminus R_b$. Then $V_{R_b}'$ meets $V_0$ but does not meet $V_1$. Thus $V_{R_b}'=V_{R_b}$. This means that $\tilde{R}_b$ lies below $\tilde{R}_a$ in $\tilde{M}_1$. Now isotope $R_a$ and $R_b$ near the boundary to form $R_a^*$ and $R_b^*$. Then it is clear that $\partial R_a^*=\partial R_b^*=\lambda_1^*$ and $R_b^*<_{\partial} R_a^*$. In addition, they continue to realise their adjacency and are $\partial$--almost disjoint.
\end{proof}

\begin{remark}\label{pushdownremark}
Suppose $R_a$ and $R_b$ are positioned as required, and that $R_b\leq R_a$. Suppose a component of $R_b$ coincides with a component of $R_a$. By Lemma \ref{makedisjointlemma} there is a unique direction we can push the component of $R_b$ off that of $R_a$ so that they continue to realise their adjacency. By examining the construction above we see that this direction is downwards. To see this, note that the components must coincide when we lift them to the infinite cyclic cover, as the boundaries of the lifts coincide.
\end{remark}

\section{Mapping the vertices}\label{verticessection}

\begin{definition}\label{psidefn}
Define a map $\Psi\colon\V(\ms(L_1))\times\V(\ms(L_2))\times\mathbb{Z}\to\V(\ms(L))$ as follows. Let $(R_1,R_2,n)\in\V(\ms(L_1))\times\V(\ms(L_2))\times\mathbb{Z}$. Take the copy of $R_1^*$ in $M_1$, and the copy of $R_2^*$ in $M_2$. Join these together by a rectangle in $M_0$ that winds $n$ times around $L$. Here we measure the winding number of the rectangle around the arc of $L$ that runs through $M_0$ and is oriented from $L_1$ to $L_2$, with respect to where $\lambda^*$ lies on the boundary of the neighbourhood of this arc. Let $R$ be the resulting surface. Note that $R$ is a taut Seifert surface for $L$. If $n\neq 0$ then $\partial R\neq\lambda^*$, but $[\partial R]=[\lambda^*]$. We set $\Psi(R_1,R_2,n)$ to be the isotopy class of $R$.
\end{definition}

\begin{lemma}\label{surjectivelemma}
$\Psi$ is surjective.
\end{lemma}
\begin{proof}
Let $R$ be a taut Seifert surface for $L$. Isotope $R$ to have minimal intersection with $T_0$, and so that $R\cap M_0=\rho\times[1,2]$ for some arc $\rho\subset T_0$. Then $T_0$ cuts $R$ into taut Seifert surfaces $R_i$ of $L_i$ for $i=1,2$. Isotope $R_i$ to $R_i^*$. This isotopy may move $\partial R_i$ in $M_i$. 
However, the isotopy of $M_i$ can be extended to an isotopy of $M$ by also isotoping the rectangle $R\cap M_0$ in $M_0$ near $T_0\times \{i\}$.
After these isotopies, we can read off the winding number $n$ of the rectangle. Then $\Psi(R_1,R_2,n)=R$.
\end{proof}

\pagebreak

\begin{lemma}\label{injectivelemma}
$\Psi$ is injective.
\end{lemma}
\begin{proof}
Suppose $\Psi(R_{a,1},R_{a,2},n_a)=\Psi(R_{b,1},R_{b,2},n_b)$. We first aim to show that $R_{a,1}=R_{b,1}$ and $R_{a,2}=R_{b,2}$. 
Let $R_a,R_b$ be the fixed surfaces constructed by $\Psi$. Note that, by construction, $R_a$ and $R_b$ each meet $T_0$ in a single arc, and these arcs either coincide or are disjoint. We can ensure that the arcs are disjoint using an isotopy of $R_a$ that breaks into an isotopy of $R_{a,1}$ and an isotopy of $R_{a,2}$. By an isotopy of $R_a$ keeping $T_0$ fixed pointwise, $R_a$ and $R_b$ can be made transverse. 

As the surfaces are isotopic, there is a product region $N$ between them. 
If $\Int(N)$ meets $T_0$ then $T_0\cap N$ is a product disc in the sutured manifold $N$, and divides $N$ into product regions $N_1,N_2$. 
The isotopy of $R_a$ defined by $N$ therefore breaks into an isotopy of $R_{a,1}$ and an isotopy of $R_{a,2}$.
It either reduces $|R_a\cap R_b|$ or makes components of $R_a$ and $R_b$ coincide.
Inductively we find that $R_{a,1}=R_{b,1}$ and $R_{a,2}=R_{b,2}$.

It remains to show that $\Psi(R_{a,1},R_{a,2},n_a)\neq\Psi(R_{a,1},R_{a,2},n_b)$ for $n_a> n_b$. Suppose otherwise. Without loss of generality, $n_b=0$. Consider the fixed surfaces $R_a=\Psi(R_{a,1},R_{a,2},n_a),R_0=\Psi(R_{a,1},R_{a,2},0)$. Push the component of the copy of $R_{a,1}$ in $R_a$  that meets $K_1$ upwards off $R_0$, and the copy of $R_{a,2}$ downwards. Delete all other components of each surface. By the assumption that $\Psi(R_{a,1},R_{a,2},n_a)=\Psi(R_{a,1},R_{a,2},0)$, the two remaining surfaces are isotopic, so there is a product region $N$ between them. Note that $N$ meets $\partial\!\nhd(K_1\# K_2)$. By considering the boundary curves of the surfaces on $\partial\!\nhd(K_1\# K_2)$, we can restrict the possibilities for the location of $N$ relative to $R_a,R_0$. To see this, note that $N$ only meets one side of each of the orientable surfaces $R_a,R_0$. Figure \ref{verticespic1} shows the boundary patterns in the cases $n_a\in\{1,2,3,4\}$.
\begin{figure}[htbp]
\centering
\input{pictexfiles/verticespic1}
\caption{\label{verticespic1}}
\end{figure}
In general we see that $N$ must meet the complement of one of the $L_i$ in the complement of the surface $R_{a,i}$. 
Suppose this is $L_1$. Then $T_0\times\{1\}$ is a product disc in $N$. This means $\Sphere\setminus\nhd(R_{a,1})$ is a product region, showing that $L_1$ is fibred, which is a contradiction.
\end{proof}

\section{Mapping the edges}\label{edgessection}

\begin{lemma}\label{threearcslemma}
Let $R_{a,i},R_{b,i}$ be fixed almost disjoint taut Seifert surfaces that demonstrate that their isotopy classes are adjacent in $\ms(L_i)$ for $i=1,2$. Suppose that there are arcs $\rho_j\subset T_0$ for $j=a,b$ such that $R_{j,i}\cap(T_0\times\{i\})=\rho_j\times\{i\}$ for $i=1,2$. Suppose further that $\rho_a$ and $\rho_b$ are disjoint.
Let $R_j=R_{j,1}\cup R_{j,2}\cup(\rho_j\times[1,2])$ for $j=a,b$.
Then $R_a,R_b$ demonstrate that their isotopy classes are adjacent in $\ms(L)$.
\end{lemma}
\begin{proof}
The pairs $R_{a,i},R_{b,i}$ satisfy the condition given in Lemma \ref{arcintersectionlemma}. We wish to show that $R_a,R_b$ also satisfy this condition. First note that $R_a$ and $R_b$ are almost disjoint.

Choose points $x,y\in R_a\setminus R_b$, an appropriate product neighbourhood of $R_a$ and a path $\rho$ from $(x,1)$ to $(y,-1)$ that is disjoint from $R_a$ and transverse to $R_b$.
By a small isotopy, $\rho$ also can be made transverse to $T_0$. If $\rho$ is disjoint from $T_0$, it has algebraic intersection 1 with $R_b$, as required.

Suppose otherwise. Let $x_0$ be the first point at which $\rho$ meets $T_0$. Find an arc $\rho_0$ from $R_a\times\{1\}$ to $R_a\times\{-1\}$ such that $\rho_0$ lies entirely in $T_0$ and passes through $x_0$. We can then use $\rho_0$ to split $\rho$ into three paths as in the condition in Lemma \ref{arcintersectionlemma}, each of which (up to isotopy) has strictly fewer points of intersection with $T_0$. The first path runs along $\rho$ as far as $x_0$, and then follows $\rho_0$ up to $R_a\times\{-1\}$. The second is $\rho_0$ traversed backwards. The third runs along $\rho_0$ to $x_0$, and then runs along the rest of $\rho$. The first two paths can be made disjoint from $T_0$, and run in opposite directions. Hence removing them does not change the algebraic intersection number with $R_b$. In this way we can remove all points of $\rho\cap T_0$.
\end{proof}

\begin{corollary}\label{edgescor}
Let $(R_{a,1},R_{a,2},n_a)$ and $(R_{b,1},R_{b,2},n_b)$ be in $\V(\ms(L_1))\times\V(\ms(L_2))\times\mathbb{Z}$ with $R_{b,2}\leq R_{a,2}$. Suppose these triples are distinct and one of the following conditions holds.
\begin{itemize}
	\item $R_{a,1}\geq R_{b,1}$ and $n_a=n_b$.
	\item $R_{a,1}\leq R_{b,1}$ and $n_a=n_b-1$.
\end{itemize}
Then $\Psi(R_{a,1},R_{a,2},n_a)$ and $\Psi(R_{b,1},R_{b,2},n_b)$ are adjacent in $\ms(L)$.
\end{corollary}
\begin{proof}
Viewing $\Psi(R_{a,1},R_{a,2},n_a)^*$ as fixed, we build a copy of $\Psi(R_{b,1},R_{b,2},n_b)$ satisfying the hypotheses of Lemma \ref{threearcslemma}. 
Without loss of generality, $n_a=0$.

Isotope $R_{b,2}^*$ as in Proposition \ref{representativesprop} to realise its adjacency with $R_{a,2}^*$. 
For any components of $R_{b,2}$ that do not coincide with those of $R_{a,2}^*$, perform a small isotopy near the boundary to move the boundary of $R_{b,2}$ to below that of $R_{a,2}^*$, making the components disjoint. 
Consider the pairs of components that meet $K_2$. 
If these components of $R_{a,2}^*,R_{b,2}$ coincide then push that in $R_{b,2}$ downwards off $R_{a,2}^*$. 
By Remark \ref{pushdownremark}, $R_{a,2}$ and $R_{b,2}$ still realise their adjacency. 

Case 1: $R_{a,1}\geq R_{b,1}$ and $n_b=n_a=0$. Then $R_{b,1}\leq_{\partial}R_{a,1}$. 
Perform isotopies on $R_{b,1}^*$ analogous to those performed on $R_{b,2}^*$.
A flat rectangle can then be inserted connecting the boundaries of the surfaces $R_{b,1},R_{b,2}$.

Case 2: $R_{a,1}\leq R_{b,1}$ and $n_b=n_a+1=1$. Then $R_{a,1}\leq_{\partial}R_{b,1}$. This time isotope $R_{b,1}^*$ upwards instead of downwards, so the boundary of $R_{b,1}$ lies above that of $R_{a,1}^*$ wherever they do not coincide. 
Add in a rectangle that wraps nearly once around $K_1\# K_2$, to again join up the boundaries of  $R_{b,1},R_{b,2}$. 

It is clear that, in either case, $\Psi(R_{b,1},R_{b,2},n_b)^*$ is isotopic to the surface $R_b$ we have constructed. We may now apply Lemma \ref{threearcslemma} to complete the proof. 
\end{proof}

\begin{proposition}\label{mainedgesprop}
Let $(R_{a,1},R_{a,2},n_a)$ and $(R_{b,1},R_{b,2},n_b)$ be in $\V(\ms(L_1))\times\V(\ms(L_2))\times\mathbb{Z}$. Suppose $\Psi(R_{a,1},R_{a,2},n_a)$ and $\Psi(R_{b,1},R_{b,2},n_b)$ are adjacent in $\ms(L)$. Then one of the conditions in Corollary \ref{edgescor} holds.
\end{proposition}
\begin{proof}
Without loss of generality, $n_a=0$.
Fix $\Psi(R_{a,1},R_{a,2},n_a)^*$ and isotope $\Psi(R_{b,1},R_{b,2},n_b)$ to be disjoint from it, realising the adjacency in $\ms(L)$. Since $T_0\cap (M_0\setminus \Psi(R_{a,1},R_{a,2},n_a)^*)$ is a disc, by standard methods we can also ensure that this copy of $\Psi(R_{b,1},R_{b,2},n_b)$ meets $T_0$ in a single arc.
As in the proof of Lemma \ref{surjectivelemma}, dividing the surface along $T_0$ gives (fixed) Seifert surfaces $R_{b,i}'$ in $\ms(L_i)$ for $i=1,2$ and there is an integer $n_b'$ such that $\Psi(R_{b,1}',R_{b,2}',n_b')$ is isotopic to $\Psi(R_{b,1},R_{b,2},n_b)$. See Figure \ref{edgespic2}. By Lemma \ref{injectivelemma} we have in particular that $R_{b,i}'$ is isotopic to $R_{b,i}$ in $M_i$ for $i=1,2$. From this we see that $R_{a,i}^*,R_{b,i}'$ demonstrate that $R_{a,i},R_{b,i}$ are at most distance 1 apart in $\ms(L_i)$. 
We may now assume that $R_{b,2}\leq R_{a,2}$. 

It now remains to verify that $n_b$ takes the required value.
Our approach is to position $\Psi(R_{b,1},R_{b,2},n_b)$ `close to' $R_{b,i}^*$ for $i=1,2$ without affecting the relative positions of $\Psi(R_{b,1},R_{b,2},n_b)$ and $\Psi(R_{a,1},R_{a,2},n_a)^*$. Having done so, we will be able to read off the value of $n_b$ from $\Psi(R_{b,1},R_{b,2},n_b)$.

For $i=1,2$, consider $R_{a,i}^*$ and $R_{b,i}^*$, which satisfy the conclusions of Proposition \ref{representativesprop}. As an isotopy of $R_{b,i}^*$ that fixes its boundary will not affect the winding number $n_b$, we may assume that no such isotopy is needed in this case. That is, we assume that $R_{a,i}^*$ and $R_{b,i}^*$ are $\partial$--almost disjoint and realise their adjacency. We further assume that components of $R_{a,i}^*$ and $R_{b,i}^*$ coincide whenever this is possible without moving the boundary of either surface.

Suppose that a component of $R_{b,i}^*$ and a component of $R_{a,i}^*$ bound a product region $N$ in $M_i$, but that any isotopy from one to the other moves the boundary of the surface. Since the boundaries of $R_{a,i}^*$ and $R_{b,i}^*$ coincide, the component of $R_{b,i}^*$ is given by taking a parallel copy of that of $R_{a,i}^*$ and moving its boundary once around $L_i$, as shown in Figure \ref{edgespic1}. 
\begin{figure}[htbp]
\centering
\psset{xunit=.80pt,yunit=.80pt,runit=.80pt}
\begin{pspicture}(280,125)
{
\newgray{lightgrey}{.8}
\newgray{lightishgrey}{.7}
\newgray{grey}{.6}
\newgray{midgrey}{.4}
\newgray{darkgrey}{.3}
}
{
\pscustom[linestyle=none,fillstyle=solid,fillcolor=lightgrey]
{
\newpath
\moveto(35.822032,24.79851)
\curveto(27.896978,25.70035)(17.537284,32.05159)(13.057354,38.75491)
\curveto(10.126341,43.14058)(8.965365,47.34846)(8.925548,53.73038)
\curveto(8.865858,63.29744)(11.76956,70.64295)(18.118604,76.98612)
\curveto(24.054623,82.91665)(30.207347,85.66347)(39.469873,86.51816)
\curveto(51.97947,87.67248)(63.55047,82.7594)(67.37822,74.66821)
\curveto(68.38895,72.5317)(68.64742,71.21305)(68.62988,68.2826)
\curveto(68.60748,64.54631)(67.36727,59.75875)(66.33703,59.4318)
\curveto(66.03607,59.3363)(65.52351,60.14708)(65.19799,61.23356)
\curveto(63.25768,67.70972)(57.7888,73.1618)(51.1842,75.20429)
\curveto(47.82647,76.24268)(41.930718,76.21672)(38.688492,75.14929)
\curveto(19.160538,68.72003)(19.075937,41.42215)(38.563656,34.86939)
\curveto(41.676484,33.8227)(47.85643,33.78701)(51.1842,34.79649)
\curveto(58.57369,37.03815)(64.3072,43.49333)(65.73685,51.18087)
\lineto(66.30491,54.23546)
\lineto(140.14517,54.23546)
\lineto(213.98543,54.23546)
\lineto(214.22123,51.78143)
\curveto(214.81341,45.61862)(219.88116,38.85034)(226.11419,35.89772)
\curveto(229.30091,34.38815)(229.87433,34.28495)(235.07462,34.28495)
\curveto(240.24328,34.28495)(240.85925,34.39397)(243.91346,35.84937)
\curveto(255.8054,41.51619)(259.71685,56.24466)(252.22932,67.16256)
\curveto(245.37828,77.15235)(230.99668,79.01552)(221.40505,71.1559)
\curveto(218.88393,69.09003)(216.06129,64.96117)(215.03853,61.84317)
\curveto(214.63419,60.61047)(214.04704,59.51646)(213.73377,59.41203)
\curveto(212.97847,59.16026)(211.0835,65.36265)(211.0835,68.08659)
\curveto(211.0835,71.99331)(212.41593,75.19364)(215.37165,78.38617)
\curveto(221.16704,84.64589)(230.1498,87.52678)(240.7465,86.52423)
\curveto(250.6551,85.58679)(258.16019,81.78867)(264.34388,74.58228)
\curveto(268.02654,70.29055)(269.75752,66.94707)(270.87636,61.96444)
\curveto(273.0456,52.30398)(270.59059,42.15875)(264.53396,35.75476)
\curveto(261.29221,32.32708)(255.65867,28.46828)(251.47664,26.81089)
\curveto(240.91885,22.62672)(229.46601,24.85226)(218.27854,33.26203)
\lineto(215.23801,35.54764)
\lineto(139.88343,35.54764)
\lineto(64.52885,35.54764)
\lineto(62.42354,33.83979)
\curveto(57.77321,30.06739)(51.16618,26.68139)(46.1761,25.51323)
\curveto(43.619452,24.91473)(38.130374,24.53582)(35.822032,24.79851)
\lineto(35.822032,24.79851)
\closepath
}
}
{
\pscustom[linewidth=1.5,linecolor=black,fillstyle=solid,fillcolor=white]
{
\newpath
\moveto(65,54.99999738)
\curveto(65,43.95430239)(56.045695,34.99999738)(45,34.99999738)
\curveto(33.954305,34.99999738)(25,43.95430239)(25,54.99999738)
\curveto(25,66.04569238)(33.954305,74.99999738)(45,74.99999738)
\curveto(56.045695,74.99999738)(65,66.04569238)(65,54.99999738)
\closepath
\moveto(255,54.99999738)
\curveto(255,43.95430239)(246.045695,34.99999738)(235,34.99999738)
\curveto(223.954305,34.99999738)(215,43.95430239)(215,54.99999738)
\curveto(215,66.04569238)(223.954305,74.99999738)(235,74.99999738)
\curveto(246.045695,74.99999738)(255,66.04569238)(255,54.99999738)
\closepath
}
}
{
\pscustom[linewidth=1,linecolor=black]
{
\newpath
\moveto(65,55)
\lineto(215,55)
\moveto(65,55)
\curveto(85,85)(35,100)(15,75)
\curveto(6.176406,63.94811)(5,45)(15,35)
\curveto(40,10)(65,35)(65,35)
\lineto(215,35)
\curveto(215,35)(240,10)(265,35)
\curveto(275,45)(275,65)(265,75)
\curveto(245,100)(195,85)(215,55)
}
}
{
\pscustom[linewidth=1,linecolor=black,linestyle=dashed,dash=3 3]
{
\newpath
\moveto(25,55)
\curveto(25,55)(20,52.071068)(20,45)
\curveto(20,35)(31,29)(45,30)
\curveto(58.48998,30.96357)(65,40)(70,45)
\lineto(210,45)
\curveto(215,40)(221,31)(235,30)
\curveto(249,29)(260,35)(260,45)
\curveto(260,52)(255,55)(255,55)
}
}
{
\put(256,60){$N$}
\put(135,65){$R_{a,i}^*$}
\put(135,20){$R_{b,i}^*$}
\put(106,42){$R_{b,i}'$}
\put(35,95){$M_i$}
\put(30,50){$\nhd(L)$}
}
\end{pspicture}
\caption{\label{edgespic1}}
\end{figure}
Note that $N$ lies below $R_{a,i}^*$ if $R_{a,i}^*\leq_{\partial}R_{b,i}^*$ and $N$ lies above $R_{a,i}^*$ if $R_{b,i}^*\leq_{\partial}R_{a,i}^*$. In addition, a component of $R_{b,i}'$ must be contained in the product region $N$ and be isotopic to the component of $R_{a,i}^*$.
In particular, from Proposition \ref{representativesprop} we see that if $R_{a,i}^*\leq R_{b,i}^*$ then the component of $R_{b,i}'$ lies above that of $R_{b,i}^*$ and its boundary can be seen as lying above that of $R_{a,i}^*$,
whereas if $R_{b,i}\leq R_{a,i}$ then the component of $R_{b,i}'$ lies below that of $R_{b,i}^*$ and its boundary can be seen as lying below that of $R_{a,i}^*$.

For $i=1,2$, temporarily ignore the components of each surface that meet $K_i$. 
If a component of $R_{b,i}^*$ coincides with a component of $R_{a,i}^*$, then there is a product region between it and the corresponding component of $R_{b,i}'$.  
Such components will play no part in the rest of the proof, so we may assume that there are none.

For each other component of $R_{b,i}^*$, 
there are two possibilities. If it is isotopic to a component of $R_{a,i}^*$, we have already seen how it relates to $R_{b,i}'$.
Otherwise, take a copy of this component of $R_{b,i}$ that lies parallel to $R_{b,i}^*$ and is disjoint from $R_{a,i}^*$. That is, there is a product region between $R_{b,i}^*$ and this copy of $R_{b,i}$ that is contained within a product neighbourhood of $R_{b,i}^*$. Furthermore, this product region lies above $R_{b,i}^*$ if $R_{a,i}^*\leq_{\partial} R_{b,i}^*$ and lies below $R_{b,i}^*$ if $R_{b,i}^*\leq_{\partial} R_{a,i}^*$.
By Lemma \ref{almostdisjointlemma}, the corresponding component of $R_{b,i}'$ is now isotopic to that of $R_{b,i}$ by an isotopy disjoint from $R_{a,i}^*$. We may therefore assume that these components of $R_{b,i}^{},R_{b,i}'$ now coincide.
Hence for each component of $R_{b,i}'$ away from $K_i$ we have the following.
If $R_{a,i}\leq R_{b,i}$ then the boundary of $R_{b,i}'$ lies above that of $R_{a,i}^*$.
If $R_{b,i}\leq R_{a,i}$ then the boundary of $R_{b,i}'$ lies below that of $R_{a,i}^*$.
Recall that we have assumed that no such components exist if $R_{a,i}=R_{b,i}$, and that $R_{b,2}\leq R_{a,2}$.

We now turn our attention to the components that meet $K_i$. 

Case 1: This component of $R_{b,i}^*$ does not coincide with $R_{a,i}^*$ for $i=1,2$. Note that this means $R_{a,i}\neq R_{b,i}$, because $R_{a,i}^*$ and $R_{b,i}^*$ are not isotopic by an isotopy fixing their boundaries. In this case, the corresponding component of $R_{b,i}'$ can also be isotoped as just described. There is now an essentially unique way to join the two surfaces by a rectangle in $M_0$ that is disjoint from $\Psi(R_{a,1},R_{a,2},n_a)^*$. If $R_{b,1}\leq R_{a,1}$ then the rectangle must be flat, so $n_b=0=n_a$. If $R_{a,1}\leq R_{b,1}$ then the rectangle must twist nearly once around $\partial\!\nhd(K_1\# K_2)$ from above $R_{a,1}^*$ to below $R_{a,2}^*$ (that is, twisting in the positive direction around $K_1\# K_2$), so $n_b=1=n_a+1$. See Figure \ref{edgespic2}.
\begin{figure}[htbp]
\centering
\psset{xunit=.55pt,yunit=.55pt,runit=.55pt}
\begin{pspicture}(540,320)
{
\newgray{lightgrey}{.8}
\newgray{lightishgrey}{.7}
\newgray{grey}{.6}
\newgray{midgrey}{.4}
\newgray{darkgrey}{.3}
}
{
\pscustom[linestyle=none,fillstyle=solid,fillcolor=lightgrey]
{
\newpath
\moveto(406.87775,72.33093)
\curveto(395.3083,74.812)(383.34201,80.41203)(370.87148,89.18128)
\curveto(362.86992,94.80796)(342.99209,114.3318)(335.86349,123.56584)
\lineto(330.87297,130.0303)
\lineto(330.87297,144.6935)
\lineto(330.87297,159.3567)
\lineto(340.04121,159.3567)
\lineto(349.20945,159.3567)
\lineto(350.99584,154.30849)
\curveto(354.07235,145.61444)(361.23746,131.22243)(366.44996,123.26696)
\curveto(372.69315,113.7384)(384.3965,101.82879)(391.82535,97.44435)
\curveto(402.26509,91.28292)(413.3502,88.46653)(423.5606,89.38139)
\curveto(434.47576,90.3594)(440.54762,92.72332)(445.80014,98.03976)
\curveto(449.6034,101.8893)(452.45738,107.55303)(456.53595,119.34501)
\curveto(460.89043,131.93466)(464.88558,139.63487)(470.0075,145.30989)
\curveto(474.38388,150.15888)(481.37956,155.27549)(486.61939,157.45978)
\curveto(489.378,158.60973)(489.4241,158.60116)(489.43357,156.93612)
\curveto(489.43857,156.00509)(490.21533,153.57413)(491.15905,151.53397)
\lineto(492.87491,147.8246)
\lineto(489.09276,145.23906)
\curveto(475.95455,136.25759)(470.22995,128.08401)(458.77452,101.95072)
\curveto(450.56128,83.21384)(443.32441,75.34935)(431.44998,72.25646)
\curveto(425.8986,70.8105)(413.7966,70.84718)(406.87775,72.33096)
\lineto(406.87775,72.33095)
\closepath
\moveto(188.8917,165.90066)
\curveto(187.93006,168.67718)(185.17224,175.08314)(182.76319,180.13614)
\curveto(169.91032,207.09522)(152.2723,224.24987)(132.24135,229.27344)
\curveto(126.23479,230.77983)(113.89622,230.95073)(108.33333,229.60459)
\curveto(102.86031,228.28019)(96.27481,224.55951)(93.29714,221.10943)
\curveto(89.73965,216.98754)(87.87133,213.1223)(83.69204,201.23809)
\curveto(77.71261,184.23504)(73.97425,177.83656)(65.88237,170.75554)
\curveto(62.2857,167.60818)(52.43346,161.60034)(50.86877,161.60034)
\curveto(50.62042,161.60034)(50.41394,162.35757)(50.40992,163.28308)
\curveto(50.40592,164.20858)(49.59799,166.69391)(48.61455,168.80601)
\lineto(46.82648,172.64622)
\lineto(48.64855,173.47641)
\curveto(52.09783,175.048)(58.40286,180.16968)(62.6059,184.8142)
\curveto(68.38842,191.20408)(71.97023,197.39841)(79.15906,213.44096)
\curveto(86.33809,229.46164)(89.35664,234.68119)(94.2921,239.60832)
\curveto(108.90392,254.1955)(138.50338,251.55704)(165.52768,233.25846)
\curveto(177.8262,224.93093)(193.56505,209.92256)(203.92492,196.64332)
\lineto(208.96821,190.17885)
\lineto(208.96821,175.51566)
\lineto(208.96821,160.85246)
\lineto(199.80417,160.85246)
\lineto(190.64013,160.85246)
\lineto(188.8917,165.90066)
\closepath
\moveto(475.2395,262.08752)
\curveto(463.79764,265.95123)(457.50204,271.81213)(453.8741,281.97767)
\curveto(451.34149,289.07407)(449.81417,290.5359)(444.93247,290.5359)
\curveto(440.2086,290.5359)(438.01291,288.08815)(433.72526,278.04216)
\curveto(432.36489,274.8548)(430.292,270.72711)(429.11882,268.86953)
\curveto(427.0047,265.52205)(422.24749,261.14326)(421.12965,261.51588)
\curveto(420.81017,261.62237)(420.3724,262.76169)(420.15684,264.04769)
\curveto(419.94128,265.3337)(419.35557,267.30209)(418.85526,268.4219)
\curveto(417.97755,270.3864)(417.98156,270.47803)(418.96935,271.03083)
\curveto(422.55131,273.03539)(427.7086,279.19164)(432.89853,287.65809)
\curveto(438.70667,297.13301)(442.15194,300.00596)(446.63177,299.10999)
\curveto(449.03569,298.62921)(450.55425,296.81591)(451.97057,292.735)
\curveto(455.93359,281.31614)(461.63477,275.74566)(473.15251,272.03863)
\curveto(475.04309,271.43015)(477.43827,270.78354)(478.47514,270.60172)
\curveto(481.58609,270.05622)(481.71243,269.82137)(480.5461,266.75185)
\curveto(479.96992,265.23547)(479.4985,263.28342)(479.4985,262.41397)
\curveto(479.4985,260.56987)(479.69026,260.58457)(475.2395,262.08752)
\closepath
}
}
{
\pscustom[linewidth=2,linecolor=black,linestyle=dashed,dash=4 4]
{
\newpath
\moveto(210,310)
\lineto(210,10)
\moveto(330,310)
\lineto(330,10)
}
}
{
\pscustom[linewidth=1.5,linecolor=midgrey]
{
\newpath
\moveto(415,270)
\curveto(431.29669,276.08791)(433.46148,300.19231)(445,300)
\curveto(454.99861,299.83335)(450,290)(460,280)
\curveto(467.90569,272.0943)(485,270)(485,270)
\moveto(490,160)
\curveto(440,140)(470,90)(420,90)
\curveto(370,90)(350,160)(350,160)
\moveto(50,160)
\curveto(100,180)(70,230)(120,230)
\curveto(170,230)(190,160)(190,160)
\moveto(420,260)
\curveto(436.29669,266.08791)(433.46148,290.19231)(445,290)
\curveto(454.99861,289.83335)(450,280)(460,270)
\curveto(467.90569,262.0943)(480,260)(480,260)
}
}
{
\pscustom[linewidth=1.5,linecolor=black]
{
\newpath
\moveto(500,150)
\curveto(450,130)(470,70)(420,70)
\curveto(370,70)(330,130)(330,130)
\lineto(270,10)
\moveto(40,170)
\curveto(90,190)(70,250)(120,250)
\curveto(170,250)(210,190)(210,190)
\lineto(270,310)
\moveto(50,160)
\lineto(490,160)
\moveto(420,260)
\lineto(480,260)
\moveto(55,40)
\lineto(115,40)
}
}
{
\pscustom[linewidth=2,linecolor=midgrey,fillstyle=solid,fillcolor=white]
{
\newpath
\moveto(230,159.99999738)
\curveto(230,148.95430239)(221.045695,139.99999738)(210,139.99999738)
\curveto(198.954305,139.99999738)(190,148.95430239)(190,159.99999738)
\curveto(190,171.04569238)(198.954305,179.99999738)(210,179.99999738)
\curveto(221.045695,179.99999738)(230,171.04569238)(230,159.99999738)
\closepath
\moveto(350,159.99999738)
\curveto(350,148.95430239)(341.045695,139.99999738)(330,139.99999738)
\curveto(318.954305,139.99999738)(310,148.95430239)(310,159.99999738)
\curveto(310,171.04569238)(318.954305,179.99999738)(330,179.99999738)
\curveto(341.045695,179.99999738)(350,171.04569238)(350,159.99999738)
\closepath
}
}
{
\pscustom[linewidth=2,linecolor=black,fillstyle=solid,fillcolor=white]
{
\newpath
\moveto(530,159.99999738)
\curveto(530,148.95430239)(521.045695,139.99999738)(510,139.99999738)
\curveto(498.954305,139.99999738)(490,148.95430239)(490,159.99999738)
\curveto(490,171.04569238)(498.954305,179.99999738)(510,179.99999738)
\curveto(521.045695,179.99999738)(530,171.04569238)(530,159.99999738)
\closepath
\moveto(50,159.99999738)
\curveto(50,148.95430239)(41.045695,139.99999738)(30,139.99999738)
\curveto(18.954305,139.99999738)(10,148.95430239)(10,159.99999738)
\curveto(10,171.04569238)(18.954305,179.99999738)(30,179.99999738)
\curveto(41.045695,179.99999738)(50,171.04569238)(50,159.99999738)
\closepath
\moveto(420,259.99999738)
\curveto(420,248.95430239)(411.045695,239.99999738)(400,239.99999738)
\curveto(388.954305,239.99999738)(380,248.95430239)(380,259.99999738)
\curveto(380,271.04569238)(388.954305,279.99999738)(400,279.99999738)
\curveto(411.045695,279.99999738)(420,271.04569238)(420,259.99999738)
\closepath
\moveto(520,259.99999738)
\curveto(520,248.95430239)(511.045695,239.99999738)(500,239.99999738)
\curveto(488.954305,239.99999738)(480,248.95430239)(480,259.99999738)
\curveto(480,271.04569238)(488.954305,279.99999738)(500,279.99999738)
\curveto(511.045695,279.99999738)(520,271.04569238)(520,259.99999738)
\closepath
\moveto(55,39.99999738)
\curveto(55,28.95430239)(46.045695,19.99999738)(35,19.99999738)
\curveto(23.954305,19.99999738)(15,28.95430239)(15,39.99999738)
\curveto(15,51.04569238)(23.954305,59.99999738)(35,59.99999738)
\curveto(46.045695,59.99999738)(55,51.04569238)(55,39.99999738)
\closepath
\moveto(155,39.99999738)
\curveto(155,28.95430239)(146.045695,19.99999738)(135,19.99999738)
\curveto(123.954305,19.99999738)(115,28.95430239)(115,39.99999738)
\curveto(115,51.04569238)(123.954305,59.99999738)(135,59.99999738)
\curveto(146.045695,59.99999738)(155,51.04569238)(155,39.99999738)
\closepath
}
}
{
\put(90,140){$R_{a,1}^*$}
\put(380,170){$R_{a,2}^*$}
\put(90,200){$R_{b,1}^*$}
\put(380,110){$R_{b,2}^*$}
\put(140,250){$R_{b,1}'$}
\put(460,80){$R_{b,2}'$}
}
\end{pspicture}
\caption{\label{edgespic2}}
\end{figure}

Case 2: The component of $R_{b,2}^*$ coincides with that of $R_{a,2}^*$, but those of $R_{a,1}^*,R_{b,1}^*$ do not coincide. Then $R_{a,1}\neq R_{b,1}$. Again isotope $R_{b,1}'$ as above. 
The component of $R_{b,2}'$ that meets $K_2$ is isotopic in $M_2$ to that of $R_{a,2}^*$, meaning there is a product region $N$ between them in $M_2$.
Note that the interior of $N$ is disjoint from the three surfaces $R_{a,2}^*,R_{b,2}^*,R_{b,2}'$. Once we know which side of $R_{a,2}^*$ the product region $N$ lies (this is well-defined, by Lemma \ref{makedisjointlemma}), there is only one possibility for the rectangle joining the surfaces $R_{b,1}',R_{b,2}'$, as in Case 1. 

If $R_{a,2}=R_{b,2}$ then we may further assume that $R_{b,1}\leq R_{a,1}$, and in particular that the boundary of $R_{b,1}'$ lies below that of $R_{a,1}^*$. If $N$ lies below $R_{a,2}^*$ then the rectangle is flat, so $n_b=0=n_a$. If $N$ lies above $R_{a,2}^*$ then the rectangle runs in the negative direction around $K_1\# K_2$, from below $R_{a,1}^*$ to above $R_{a,2}^*$. Hence $n_b=-1=n_a-1$. Note that this satisfies the condition in Corollary \ref{edgescor} with the two surfaces switched. 

Suppose instead that $R_{a,2}\neq R_{b,2}$. If $N$ lies below $R_{a,2}^*$ then there are two possible situations. One is that $R_{b,1}\leq R_{a,1}$ and $n_b=n_a$. The other is that $R_{a,1}\leq R_{b,1}$ and $n_b=n_a+1$.
It therefore remains to rule out the possibility that $N$ lies above $R_{a,2}^*$. 
For this we must return to the definitions of $R_{a,2}^*,R_{b,2}^*$. Originally we chose these to be $\partial$--almost disjoint, with $R_{b,2}^*\leq_{\partial} R_{a,2}^*$. As they now coincide, the components that meet $K_2$ were isotopic by an isotopy keeping their boundaries fixed. Corollary \ref{isotopyrelbdycor} therefore shows there was a product region $N'$ between these components that lay below $R_{a,2}^*$ and above $R_{b,2}^*$.
However, Lemma \ref{makedisjointlemma} says that the side of $R_{a,2}^*$ on which $N$ lies is determined by the choice of surfaces $R_{a,2},R_{b,2}$. Therefore $N$ cannot lie above $R_{a,2}^*$.

Case 3: The component of $R_{b,1}$ coincides with that of $R_{a,1}^*$, but those of $R_{a,2}^*,R_{b,2}$ do not coincide. This is similar to case 2.

Case 4: Both pairs of components coincide. 
This case again uses the same ideas as cases 1 and 2. The only situation that is very different is when $R_{a,1}=R_{b,1}$ and $R_{a_2}=R_{b,2}$, when we must show that $n_a\neq n_b$. However this is true since $\Psi(R_{a,1},R_{a,2},n_a)\neq\Psi(R_{b,1},R_{b,2},n_b)$.
\end{proof}

Figure \ref{edgespic3} is a schematic picture of the local structure of $\ms(L)$. Figure \ref{edgespic3}a focuses on the edges radiating from a single vertex. Figure \ref{edgespic3}b shows one of the smaller cubes in Figure \ref{edgespic3}a, with all edges included.
\begin{figure}[htbp]
\centering
(a)
\psset{xunit=.5pt,yunit=.5pt,runit=.5pt}
\begin{pspicture}(310,300)
{
\pscustom[linewidth=1.5,linecolor=gray]
{
\newpath
\moveto(270,90.00001)
\lineto(180,90.00001)
\lineto(90,90.00001)
\lineto(90,180.00003)
\lineto(90,270.00003)
\lineto(180,270.00003)
\lineto(270,270.00003)
\lineto(270,180.00003)
\lineto(270,90.00001)
\closepath
\moveto(180,270.00003)
\lineto(180,90.00001)
\moveto(270,180.00003)
\lineto(90,180.00003)
\moveto(240,110.00001)
\lineto(150,110.00001)
\lineto(60,110.00001)
\lineto(60,200.00003)
\lineto(60,290.00003)
\lineto(150,290.00003)
\lineto(240,290.00003)
\lineto(240,200.00003)
\lineto(240,110.00001)
\closepath
\moveto(150,290.00003)
\lineto(150,110.00001)
\moveto(240,200.00003)
\lineto(60,200.00003)
\moveto(300,250.00001)
\lineto(240,290.00001)
\moveto(210,250.00001)
\lineto(150,290.71429)
\moveto(120,250.00001)
\lineto(60,290.71429)
\moveto(120,160.00001)
\lineto(60,200.00001)
\moveto(120,69.99999)
\lineto(60,110.00001)
\moveto(210,69.99999)
\lineto(150,110.00001)
\moveto(300,69.99999)
\lineto(240,110.00001)
\moveto(300,160.00001)
\lineto(240,200.00001)
\moveto(210,160.00001)
\lineto(150,200.00001)
}
}
{
\pscustom[linewidth=2,linecolor=black]
{
\newpath
\moveto(180,180.00001)
\lineto(90,180.00001)
\moveto(150,200)
\lineto(180,180)
\moveto(210,160.00001)
\lineto(180,180.00001)
\moveto(270,180.00001)
\lineto(180,180.00001)
\moveto(60,200.00001)
\lineto(180,180.00001)
\moveto(270,270.00001)
\lineto(180,180.00001)
\moveto(180,270.00001)
\lineto(180,180.00001)
\moveto(150,290.00001)
\lineto(180,180.00001)
\moveto(240,290.00001)
\lineto(180,180.00001)
\moveto(180,89.99999)
\lineto(180,180.00001)
\moveto(300,160.00001)
\lineto(180,180.00001)
\moveto(180,180.00001)
\lineto(210,69.99999)
\moveto(90,90.00001)
\lineto(180,180.00001)
\moveto(120,70.00001)
\lineto(180,180.00001)
}
}
{
\pscustom[linewidth=1.5,linecolor=gray]
{
\newpath
\moveto(60,110.00001)
\lineto(60,290.00001)
\lineto(240,290.00001)
\moveto(90,90.00001)
\lineto(90,270.00001)
\lineto(270,270.00001)
\moveto(300,69.99999)
\lineto(210,69.99999)
\lineto(120,69.99999)
\lineto(120,160.00001)
\lineto(120,250.00001)
\lineto(210,250.00001)
\lineto(300,250.00001)
\lineto(300,160.00001)
\lineto(300,69.99999)
\closepath
\moveto(210,250.00001)
\lineto(210,69.99999)
\moveto(300,160.00001)
\lineto(120,160.00001)
}
}
{
\pscustom[linewidth=1,linecolor=black]
{
\newpath
\moveto(40,90.00001)
\lineto(100,50.00001)
\moveto(120,40.00001)
\lineto(300,40.00001)
\moveto(30,110.00001)
\lineto(30,290.00001)
}
}
{
\pscustom[linewidth=3,linecolor=black,fillstyle=solid,fillcolor=black]
{
\newpath
\moveto(91.67949706,55.54701196)
\lineto(86.13249509,54.43761157)
\lineto(100,50.00001)
\lineto(90.57009666,61.09401392)
\lineto(91.67949706,55.54701196)
\closepath
\moveto(290,40.00001)
\lineto(286,36.00001)
\lineto(300,40.00001)
\lineto(286,44.00001)
\lineto(290,40.00001)
\closepath
\moveto(30,280.00001)
\lineto(34,276.00001)
\lineto(30,290.00001)
\lineto(26,276.00001)
\lineto(30,280.00001)
\closepath
}
}
{
\put(5,210){$n$}
\put(45,45){$R_2$}
\put(170,15){$R_1$}
}
\end{pspicture}
(b)
\psset{xunit=.5pt,yunit=.5pt,runit=.5pt}
\begin{pspicture}(310,300)
{
\pscustom[linewidth=1.5,linecolor=gray]
{
\newpath
\moveto(300,69.999977)
\lineto(240,109.999997)
\moveto(120,69.999997)
\lineto(240,289.999997)
\moveto(240,289.999997)
\lineto(300,69.999997)
\moveto(240,109.999997)
\lineto(150,109.999997)
\lineto(60,109.999997)
\lineto(60,200.000017)
\lineto(60,290.000017)
\lineto(150,290.000017)
\lineto(240,290.000017)
\lineto(240,200.000017)
\lineto(240,109.999997)
\closepath
\moveto(300,69.999997)
\lineto(60,109.999997)
\lineto(240,288.571427)
}
}
{
\pscustom[linewidth=2,linecolor=black]
{
\newpath
\moveto(300,69.999977)
\lineto(210,69.999977)
\lineto(120,69.999977)
\lineto(120,159.999997)
\lineto(120,249.999997)
\lineto(210,249.999997)
\lineto(300,249.999997)
\lineto(300,159.999997)
\lineto(300,69.999977)
\closepath
\moveto(300,249.999997)
\lineto(240,289.999997)
\lineto(60,289.999997)
\moveto(120,249.999997)
\lineto(60,289.999997)
\moveto(120,69.999977)
\lineto(60,109.999997)
\lineto(60,289.999997)
\moveto(300,249.999997)
\lineto(120,69.999997)
\moveto(60,289.999997)
\lineto(120,69.999997)
\moveto(60,289.999997)
\lineto(300,249.999997)
}
}
{
\pscustom[linewidth=1,linecolor=black]
{
\newpath
\moveto(40,90.00001)
\lineto(100,50.00001)
\moveto(120,40.00001)
\lineto(300,40.00001)
\moveto(30,110.00001)
\lineto(30,290.00001)
}
}
{
\pscustom[linewidth=3,linecolor=black,fillstyle=solid,fillcolor=black]
{
\newpath
\moveto(91.67949706,55.54701196)
\lineto(86.13249509,54.43761157)
\lineto(100,50.00001)
\lineto(90.57009666,61.09401392)
\lineto(91.67949706,55.54701196)
\closepath
\moveto(290,40.00001)
\lineto(286,36.00001)
\lineto(300,40.00001)
\lineto(286,44.00001)
\lineto(290,40.00001)
\closepath
\moveto(30,280.00001)
\lineto(34,276.00001)
\lineto(30,290.00001)
\lineto(26,276.00001)
\lineto(30,280.00001)
\closepath
}
}
{
\put(5,210){$n$}
\put(45,45){$R_2$}
\put(170,15){$R_1$}
}
\end{pspicture}
\caption{\label{edgespic3}}
\end{figure}

\section{Completing and interpreting the proof}\label{proofsection}

\connectsumthm*
\begin{proof}
Define an ordering $\leq_1$ on $\V(\ms(L_1))\times\V(\mathcal{Z})$ by 
\[(R_b,n_b)\leq_1 (R_a,n_a)\iff \left(R_b\leq R_a\textrm{ and }n_b\leq n_a\right).\]
By Theorem \ref{productcomplextheorem}, 
\[{\big|(\ms(L_1)\times\mathcal{Z},\leq_1)\big|}\cong{\big|(\ms(L_1),\leq)\big|\times\big|(\mathcal{Z},\leq)\big|}.\]
Now define a second ordering $\leq_2$ on $\V(\ms(L_1))\times\V(\mathcal{Z})$ by 
\[
\begin{split}
(R_b,n_b)\leq_2 &(R_a,n_a)\iff\\
&\big((R_b\leq R_a\textrm{ and }n_b=n_a)\textrm{ or }(R_a\leq R_b\textrm{ and }n_a=n_b-1)\big).
\end{split}
\] 
Note that this corresponds to the conditions in Corollary \ref{edgescor}. It can be checked that $\leq_2$ has properties (P1), (P2)$'$, (P3).

Define $\leq_3$ on $\V(\ms(L_1)\times\mathcal{Z})\times\V(\ms(L_2))$ by 
\[
\begin{split}
\big((R_{a,1},n_a),R_{a,2}\big)\leq_3&\big((R_{b,1},n_b),R_{b,2}\big)\iff\\
&\big((R_{a,1},n_a)\leq_2 (R_{b,1},n_b)\textrm{ and }R_{a,2}\leq R_{b,2}\big).
\end{split}
\]
Applying Theorem \ref{productcomplextheorem} again gives that 
\[{\big|\big((\ms(L_1)\!\times\mathcal{Z})\!\times\ms(L_2),\leq_3\!\!\hspace{0.05em}\big)\big|}\cong{\big|(\ms(L_1)\times\mathcal{Z},\leq_2)\big|\times\big|(\ms(L_2),\leq)\big|}.\]
Thus
\[|(\ms(L_1)\!\times\mathcal{Z})\!\times\ms(L_2)|\cong|\ms(L_1)|\times|\ms(L_2)|\times\mathbb{R}.\]

By Lemmas \ref{surjectivelemma} and \ref{injectivelemma}, the map $\psi\colon\V(\ms(L_1))\times\V(\ms(L_2))\times\V(\mathcal{Z})\to\V(\ms(L))$ defined in Definition \ref{psidefn} is a bijection. Recall that 
\[\V\!\big((\ms(L_1)\!\times\mathcal{Z})\!\times\ms(L_2)\big)=\V(\ms(L_1))\times\V(\ms(L_2))\times\V(\mathcal{Z}).\]
From Corollary \ref{edgescor} and Proposition \ref{mainedgesprop} we see that $\psi$ extends to an isomorphism between the 1--skeleta of the complexes $(\ms(L_1)\!\times\mathcal{Z})\!\times\ms(L_2)$ and $\ms(L)$.
It remains only to note that both of these complexes are flag. For $\ms(L)$ this is the case by definition. For $(\ms(L_1)\!\times\mathcal{Z})\!\times\ms(L_2)$ it follows from the fact that the three complexes $\ms(L_1)$, $\ms(L_2)$ and $\mathcal{Z}$ are flag.
\end{proof}

By examining the proof of Theorem \ref{productcomplextheorem} in \cite{MR0050886} we can give the following geometric description of the extension of $\Psi$ to $(\ms(L_1)\!\times\mathcal{Z})\!\times\ms(L_2)$.

\begin{remark}
Let $x\in|\ms(L_1)|\times|\ms(L_2)|\times\mathbb{R}$. Without loss of generality, $\pi_{\mathcal{Z}}(x)\in[0,1)$.
Let $\pi_{\ms(L_1)}(x)=a_0 A_0 +\cdots+a_m A_m$ where $A_i\in\V(\ms(L_1))$ and $a_i>0$ for $0\leq i\leq m$, with $\sum_{i=0}^m{a_i}=1$ and $A_0\leq A_1\leq\cdots\leq A_m$. Similarly let $\pi_{\ms(L_2)}(x)=b_0 B_0 +\cdots+b_n B_n$.

Consider the surfaces $A_0,\ldots,A_m$. As in Lemma \ref{maketransverselemma}, 
they can be positioned in $M_1$ so they are pairwise almost disjoint with simplified intersection. By Proposition \ref{realisedistanceprop}, they then realise their adjacencies. As in the proof of Lemma \ref{makedisjointlemma}, the surfaces can be made disjoint while still realising their adjacencies, and it can be shown that the boundaries of the surfaces occur in order around $\partial M_1$. For $i=0,\ldots,m$, thicken the Seifert surface $A_i$ to a product region $A_i\times [0,a_i]$, and view this as a `continuum of surfaces'.

Do the same for the Seifert surfaces $B_0,\ldots,B_n$ in $M_2$. Glue the thickened surfaces to give thickened Seifert surfaces for $L$ in $M$. In doing so, instead of aligning $A_0\times\{0\}$ with $B_0\times\{0\}$, introduce a shift of length $\pi_{\mathcal{Z}}(x)$. This creates a finite set of vertices of $\ms(L)$, each with a weight given by its thickness.
Applying Lemma \ref{threearcslemma} shows that these vertices span a simplex.
\end{remark}

\section{Incompressible surfaces}\label{incompressiblesection}

In addition to $\ms(L)$, Kakimizu defined a larger complex $\is(L)$, which records all incompressible Seifert surfaces for $L$ rather than just taut ones.

\begin{definition}[see \cite{MR1177053}, \cite{przytycki-2010}]
Let $L$ be a link, and let $M=\Sphere\setminus\nhd(L)$. Define $\is(L)$ of $L$ to be the following flag simplicial complex. Its vertices are ambient isotopy classes of incompressible Seifert surfaces for $L$. Two vertices span an edge if they have representatives $R,R'$ such that a lift of $M\setminus R'$ to the infinite cyclic cover of $M$ intersects exactly two lifts of $M\setminus R$.
\end{definition}

Note that $\ms(L)$ is a subcomplex of $\is(L)$.
Proposition \ref{msdistanceprop} holds for $\is(L)$ as well as for $\ms(L)$, and so do Propositions \ref{disjointinteriorsprop} and \ref{realisedistanceprop}, and Lemma \ref{maketransverselemma}. 
The same is true of Lemmas \ref{orderinglemma1}, \ref{orderdistancelemma} and \ref{orderinglemma2}. 

Let $R$ be an incompressible Seifert surface for $L$. Isotope $R$ to have minimal intersection with $T_0$. Then $R\cap T_0$ is a single arc. Splitting $R$ along $T_0$ gives incompressible Seifert surfaces $R_1,R_2$ for $L_1,L_2$ respectively.

Now consider the converse situation. That is, take incompressible Seifert surfaces $R_1,R_2$ for $L_1,L_2$ respectively, and join them along an arc in $T_0$ to form a Seifert surface $R$ for $L$.

\begin{lemma}
$R$ is incompressible.
\end{lemma}
\begin{proof}
Suppose otherwise. Choose a compressing disc $S$ for $R$ that minimises its intersection with $T_0$ over all compressing discs for $R$. Then $S\cap T_0$ does not include any simple closed curves. In addition, $S$ is not disjoint from $T_0$, as otherwise it would be a compressing disc for either $R_1$ or $R_2$. Let $\rho$ be an arc of $S\cap T_0$ that is outermost in the disc $T_0\cap (M_0\setminus R)$. Then $\rho$ cuts off a subdisc $S_T$ of $T_0\cap (M_0\setminus R)$ that is disjoint on its interior from $S$. It also divides $S$ into two discs $S_1$ and $S_2$. Since $|S\cap T_0|$ cannot be reduced, each of the discs $S_1\cup S_T$ and $S_2\cup S_T$ is a compressing disc for $R$. Furthermore, each can be isotoped to have a smaller intersection with $T_0$ than $S$ does, which is a contradiction.
\end{proof}

Now replacing taut Seifert surfaces with incompressible Seifert surfaces in the proof of Theorem \ref{connectsumthm} gives the following.

\begin{theorem}
Let $L_1,L_2$ be non-split, non-fibred, links in $\Sphere$, 
and let $L=L_1\#L_2$. Then $|\is(L)|$ is homeomorphic to $|\is(L_1)|\times|\is(L_2)|\times\mathbb{R}$.
\end{theorem}

\bibliography{connectedsumsrefs}
\bibliographystyle{hplain}

\bigskip
\noindent
Mathematical Institute

\noindent
University of Oxford

\noindent
24--29 St Giles'

\noindent
Oxford OX1 3LB

\noindent
England

\smallskip
\noindent
\textit{jessica.banks[at]lmh.oxon.org}

\end{document}